\documentclass[11pt]{amsart}
\usepackage{amsopn, amssymb, amscd, amsmath, amsthm}
\usepackage{graphicx, graphics, epsfig}
\usepackage{enumerate}

\usepackage{todonotes} 

\usepackage{lscape, pdflscape}  
\usepackage{array, multirow} 
\usepackage{afterpage, capt-of} 

\newcommand{\nc}{\newcommand}

\nc{\fg}{\mathfrak{f}}  \nc{\vg}{\mathfrak{v}} \nc{\wg}{\mathfrak{w}} \nc{\zg}{\mathfrak{z}} \nc{\ngo}{\mathfrak{n}} \nc{\kg}{\mathfrak{k}} \nc{\mg}{\mathfrak{m}} \nc{\bg}{\mathfrak{b}} \nc{\ggo}{\mathfrak{g}} \nc{\ggob}{\overline{\mathfrak{g}}} \nc{\sog}{\mathfrak{so}} \nc{\sug}{\mathfrak{su}} \nc{\spg}{\mathfrak{sp}} \nc{\slg}{\mathfrak{sl}} \nc{\glg}{\mathfrak{gl}} \nc{\cg}{\mathfrak{c}} \nc{\rg}{\mathfrak{r}}  \nc{\hg}{\mathfrak{h}} \nc{\tgo}{\mathfrak{t}} \nc{\ug}{\mathfrak{u}} \nc{\dg}{\mathfrak{d}} \nc{\ag}{\mathfrak{a}} \nc{\pg}{\mathfrak{p}} \nc{\sg}{\mathfrak{s}} \nc{\affg}{\mathfrak{aff}} \nc{\qg}{\mathfrak{q}}
\nc{\Xg}{\mathfrak{X}} \nc{\lgo}{\mathfrak{l}}

\nc{\pca}{\mathcal{P}} \nc{\nca}{\mathcal{N}} \nc{\lca}{\mathcal{L}} \nc{\oca}{\mathcal{O}} \nc{\mca}{\mathcal{M}} \nc{\tca}{\mathcal{T}} \nc{\aca}{\mathcal{A}} \nc{\cca}{\mathcal{C}} \nc{\gca}{\mathcal{G}} \nc{\sca}{\mathcal{S}} \nc{\hca}{\mathcal{H}} \nc{\bca}{\mathcal{B}} \nc{\dca}{\mathcal{D}}

\nc{\vp}{\varphi} \nc{\ddt}{\tfrac{{\rm d}}{{\rm d}t}} \nc{\dds}{\tfrac{{\rm d}}{{\rm d}s}} \nc{\ddtbig}{\frac{{\rm d}}{{\rm d}t}} \nc{\dd}{{\rm d}}
\nc{\dpar}{\tfrac{\partial}{\partial t}} \nc{\im}{\mathtt{i}}

\nc{\SO}{\mathrm{SO}} \nc{\Spe}{\mathrm{Sp}} \nc{\Sl}{\mathrm{SL}}
\nc{\SU}{\mathrm{SU}} \nc{\Or}{\mathrm{O}} \nc{\U}{\mathrm{U}} \nc{\Gl}{\mathrm{GL}}
\nc{\Se}{\mathrm{S}} \nc{\Cl}{\mathrm{Cl}} \nc{\Spein}{\mathrm{Spin}}
\nc{\Pin}{\mathrm{Pin}} \nc{\G}{\mathrm{GL}_n(\RR)} \nc{\g}{\mathfrak{gl}_n(\RR)}

\nc{\RR}{{\mathbb R}} \nc{\HH}{{\mathbb H}} \nc{\CC}{{\mathbb C}} \nc{\ZZ}{{\mathbb Z}}
\nc{\FF}{{\mathbb F}} \nc{\NN}{{\mathbb N}} \nc{\QQ}{{\mathbb Q}} \nc{\PP}{{\mathbb P}}

\nc{\vs}{\vspace{.2cm}} \nc{\vsp}{\vspace{1cm}} \nc{\ip}{{\langle\cdot,\cdot\rangle}}
\nc{\ipp}{(\cdot,\cdot)} \nc{\la}{\langle} \nc{\ra}{\rangle} \nc{\unm}{\tfrac{1}{2}}
\nc{\unc}{\tfrac{1}{4}} \nc{\und}{\tfrac{1}{16}} \nc{\no}{\vs\noindent}
\nc{\lam}{\Lambda^2(\RR^n)^*\otimes\RR^n} \nc{\tangz}{{\rm T}^{\rm Zar}}
\nc{\lamg}{\Lambda^2\ggo^*\otimes\ggo}
\nc{\nor}{{\sf n}}  \nc{\mum}{/\!\!/} \nc{\kir}{/\!\!/\!\!/}
\nc{\Ri}{\tfrac{4\Ric_{\mu}}{||\mu||^2}} \nc{\ds}{\displaystyle}
\nc{\ben}{\begin{enumerate}} \nc{\een}{\end{enumerate}} \nc{\f}{\frac}
\nc{\lb}{[\cdot,\cdot]} \nc{\isn}{\tfrac{1}{||v||^2}}
\nc{\gkp}{(\ggo=\kg\oplus\pg,\ip)} \nc{\ukh}{(\ug=\kg\oplus\hg,\ip)}
\nc{\tgkp}{(\tilde{\ggo}=\kg\oplus\pg,\ip)}
\nc{\wt}{\widetilde}
\nc{\raw}{\rightarrow} \nc{\lraw}{\longrightarrow} \nc{\hqn}{\mathcal{H}_{q,n}}
\nc{\minimatrix}[4]{\left[\begin{smallmatrix} {#1} & {#2} \\ {#3} & {#4} \end{smallmatrix}\right]}
\nc{\twomatrix}[4]{\left[\begin{array}{cc} {#1} & {#2} \\ {#3} & {#4} \end{array} \right]}
\nc{\threematrix}[9]{\left[\begin{array}{ccc} {#1} & {#2} & {#3} \\ {#4} & {#5} & {#6}\\ {#7} & {#8} & {#9} \end{array} \right]}

\nc{\ad}{\operatorname{ad}}  \nc{\Aut}{\operatorname{Aut}}   \nc{\Inn}{\operatorname{Inn}}   \nc{\Lie}{\operatorname{Lie}} \nc{\Ad}{\operatorname{Ad}} \nc{\Der}{\operatorname{Der}} \nc{\rad}{\operatorname{r}} \nc{\kf}{\operatorname{B}}

\nc{\End}{\operatorname{End}} \nc{\rank}{\operatorname{rank}} \nc{\Ker}{\operatorname{Ker}} \nc{\tr}{\operatorname{tr}} \nc{\sym}{\operatorname{sym}} \nc{\diag}{\operatorname{diag}} \nc{\proy}{\operatorname{pr}} \nc{\Adj}{\operatorname{Adj}}

\nc{\Hess}{\operatorname{Hess}}  \nc{\dif}{\operatorname{d}} \nc{\sen}{\operatorname{sen}} \nc{\grad}{\operatorname{grad}} \nc{\Order}{\operatorname{O}} \nc{\divg}{\operatorname{div}}

\nc{\Iso}{\operatorname{Iso}} \nc{\Diff}{\operatorname{Diff}} \nc{\ricci}{\operatorname{Ric}}  \nc{\Rc}{\operatorname{Rc}} \nc{\Ricci}{\operatorname{Ric}} \nc{\Riem}{\operatorname{Rm}} \nc{\scalar}{\operatorname{sc}} \nc{\scalarm}{\hat{\operatorname{R}}} \nc{\riccim}{\widehat{\operatorname{Ric}}} \nc{\tang}{\operatorname{T}} \nc{\vol}{\operatorname{vol}}

\nc{\mm}{\operatorname{M}} \nc{\CH}{\operatorname{CH}} \nc{\Irr}{\operatorname{Irr}} \nc{\mcc}{\operatorname{mcc}}

\theoremstyle{plain}
\newtheorem{theorem}{Theorem}[section]
\newtheorem{proposition}[theorem]{Proposition}
\newtheorem{corollary}[theorem]{Corollary}
\newtheorem{lemma}[theorem]{Lemma}
\newtheorem{teointro}{Theorem}

\newtheorem*{AC}{Alekseevskii's conjecture}

\theoremstyle{definition}
\newtheorem{definition}[theorem]{Definition}

\theoremstyle{remark}
\newtheorem{remark}[theorem]{Remark}

\title[The Alekseevskii conjecture in low dimensions]{The Alekseevskii conjecture in low dimensions}

\author{Romina M.~Arroyo}
\address{FaMAF $\&$ CIEM, Universidad Nacional de C\'ordoba, C\'ordoba, Argentina}
\email{arroyo@famaf.unc.edu.ar}

\author{Ramiro A.~Lafuente}
\address{Mathematisches Institut, Universit\"at M\"unster, Einsteinstr.~ 62, 48149 M\"unster, Germany}
\email{lafuente@uni-muenster.de}

\thanks{This research was supported by fellowships from CONICET and grants from CONICET, FONCYT and SeCyT (Universidad Nacional de C\'ordoba).}

\thanks{The second author was supported by fellowships from CONICET and the Alexander von Humboldt Foundation.}

\begin{document}

\begin{abstract}
The long-standing Alekseevskii conjecture states that a connected homogeneous Einstein space $G/K$ of negative scalar curvature must be diffeomorphic to $\RR^n$. This was known to be true only in dimensions up to $5$, and in dimension $6$ for non-semisimple $G$. In this work we prove that this is also the case in dimensions up to $10$ when $G$ is not semisimple. For arbitrary $G$, besides $5$ possible exceptions, we show that the conjecture holds up to dimension $8$.
\end{abstract}

\maketitle

\section{Introduction}

A Riemannian manifold $(M^n, g)$ is called Einstein if its Ricci tensor satisfies $\Ricci(g) = c\, g$, for some $c\in \RR$. This is a very subtle condition, since it is too strong to allow general existence results, and at the same time too weak for obtaining obstructions in dimensions above $4$. It is therefore natural to consider the Einstein equation for a special class of metrics such as K\"ahler, Sasakian, with special holonomy, or with some symmetry assumption, among others (see \cite{LeBWang, cruzchica, Sparkssurvey, Wang2012} for further details and examples).

We study this equation on homogeneous manifolds. The classification of homogeneous Einstein spaces is naturally divided into cases according to the sign of the scalar curvature. Ricci-flat homogeneous manifolds are flat by \cite{AlkKml}. If the scalar curvature is positive, the manifold must be compact by Bonet-Myers' theorem, while a theorem of Bochner \cite{Bochner1948} implies that if it is negative, the manifold is non-compact. In the latter case, the following fundamental problem remains unsolved

\begin{AC}\cite[7.57]{Bss}
Any connected homogeneous Einstein space of negative scalar curvature is diffeomorphic to a Euclidean space.
\end{AC}

The purpose of the present article is to investigate this conjecture in low-dimensional spaces. Recall that in dimensions $2$ and $3$, Einstein metrics have constant sectional curvature. Simply-connected homogeneous Einstein $4$-manifolds were classified by G.\ Jensen in his thesis \cite{Jns}, and they are all isometric to symmetric spaces. In dimension $5$, non-compact homogeneous Einstein spaces $G/K$ were studied in \cite{Nkn1}, where it was shown that if $G\neq \Sl_2(\CC)$ then they are isometric to simply-connected Einstein solvmanifolds, and in particular diffeomorphic to $\RR^5$. In the recent work \cite{semialglow} the authors proved that the conjecture holds in dimension $6$, provided there exists a non-semisimple transitive group of isometries (a shorter proof of this fact was recently obtained in \cite{JblPtr}). Our first main result is the following

\begin{teointro}\label{main6}
Let $(M^6,g)$ be a $6$-dimensional connected homogeneous Einstein space of negative scalar curvature, on which neither $\Sl_2(\CC)$ nor $\widetilde{\Sl_2(\RR)}\times \widetilde{\Sl_2(\RR)}$ acts transitively by isometries. Then, $M^6$ is diffeomorphic to $\RR^6$.
\end{teointro}

Remarkably, the question of whether the $6$-dimensional simple Lie groups $\Sl_2(\CC)$ and $\widetilde{\Sl_2(\RR)}\times \widetilde{\Sl_2(\RR)}$ admit a left-invariant Einstein metric is still open. This is however not surprising if one recalls that the total number of homogeneous Einstein metrics on its compact counterpart $S^3 \times S^3$ is still unknown, even though the compact case has been much more investigated in the literature.

Our second main result confirms the conjecture in dimension $7$.

\begin{teointro}\label{main7}
Any $7$-dimensional connected homogeneous Einstein space of negative scalar curvature is diffeomorphic to $\RR^7$.
\end{teointro}

Besides the case of left-invariant metrics on two simple Lie groups and one very special homogeneous space, we show that the conjecture also holds in dimension $8$.

\begin{teointro}\label{main8}
Let $(M^8,g)$ be an $8$-dimensional connected homogeneous Einstein space of negative scalar curvature which is de Rham irreducible. Assume that $(M^8,g)$ is not an invariant metric on the simply connected  homogeneous space $\left(\Sl_2(\RR)\times \Sl_2(\CC)\right)/\Delta\U(1)$, and that neither $\widetilde{\Sl_3(\RR)}$ nor $\widetilde{\SU(2,1)}$ acts transitively by isometries. Then, $M^8$ is diffeomorphic to $\RR^8$.
\end{teointro}

It is important to remark that in Theorems \ref{main6}, \ref{main7} and \ref{main8} we actually obtain a stronger conclusion, namely that the spaces admit a simply-transitive solvable group of isometries (i.e.~ they are isometric to a \emph{solvmanifold}). We mention here that there is a stronger version of the conjecture, which is obtained by replacing the conclusion ``diffeomorphic to a Euclidean space'' by ``isometric to a simply-connected solvmanifold'' (this is commonly referred to as the \emph{strong Alekseevskii conjecture} in the literature, see \cite{JblPtr}). Both statements turn out to be equivalent when the isometry group is linear, and in fact at the present time all known-examples of homogeneous Einstein spaces with negative scalar curvature are isometric to simply-connected solvmanifolds.

Finally, we focus on the case where the presentation group is not semisimple. Our main result in this direction is the following

\begin{teointro}\label{mainnonuni}
Let $(M^n,g)$ be a simply-connected non-compact homogeneous Einstein space of dimension less than or equal to $10$, which is de Rham irreducible. If $(M^n,g)$ admits a non-semisimple transitive group of isometries, then $M^n$ is diffeomorphic to $\RR^n$.
\end{teointro}



Using a close link relating non-compact homogeneous Einstein spaces and expanding homogeneous Ricci solitons (cf.~\cite{HePtrWyl,alek} and \cite{Jbl13b}), Theorem \ref{mainnonuni} immediately implies the following result.

\begin{corollary}
Let $(M^n,g)$ be a simply-connected expanding homogeneous Ricci soliton which is not Einstein, of dimension less than or equal to $9$, and which is de Rham irreducible. Then, $M^n$ is diffeomorphic to $\RR^n$.
\end{corollary}


With regards to other previous known results on low-dimensional homogeneous Einstein spaces, we mention that the classification of simply-connected compact homogeneous Einstein manifolds was obtained in \cite{AleDotFer} in dimension $5$, and in \cite{Nkn04} in dimension $7$. Partial results in dimension $6$ may be found in \cite{NknRod03}. In \cite{BhmKrr} it was proved that all simply-connected compact homogeneous spaces of dimension less than $12$ admit a homogeneous Einstein metric. In the non-compact case, the classification of Einstein solvmanifolds in low dimensions was studied in \cite{finding, Wll03, NikitenkoNikonorov, FC13}.

The starting point for the proof of our main results are the structural results for non-compact homogeneous Einstein spaces given in \cite{alek}, and specially its more recent refinements proved in \cite{JblPtr}. Roughly speaking, these results state that the simply-connected cover of such a space admits a very special presentation of the form $G/K$, where $G = \left( G_1 A\right)\ltimes N$ is a semi-direct product of a nilpotent normal Lie subgroup $N$ and a reductive Lie subgroup $U = G_1 A$, with center $A$ and whose semisimple part $G_1 = [U,U]$ has no compact simple factors and contains the isotropy $K$. Moreover, the orbits of $U$ and $N$ are orthogonal at $eK$, the induced metric on $N$ is a homogeneous Ricci soliton, and the induced metric on $U/K$ satisfies an Einstein-like condition in which the action of $U$ on $N$ comes into play (see \eqref{eqRicU/K} below). In the present article we further improve those results by showing that the orbits of $A$ and $G_1$ are also orthogonal at $eK$ (Theorem \ref{thm_lemadimn}). This allows us to reduce the problem to solving the generalized Einstein equation \eqref{eqRicU/K} on $G_1/K$, which turns out to be a homogeneous space of dimension at most $7$ with semisimple transitive group. Moreover, as an application of our new structure refinements we present a short proof of a result of Jablonski \cite{Jbl13b} which states that homogeneous Ricci solitons are algebraic (Corollary~ \ref{cor_algebraic}).

The reduction to the simply-connected case is possible in dimensions $8$ and lower because we show that those spaces are isometric to solvmanifolds, thus allowing us to apply the results in \cite{Jab15}.

In order to study the Einstein equation (and its generalized version) in the semisimple case, we give in Table \ref{tabla} a complete classification of non-compact homogeneous spaces with a semisimple transitive group without compact simple factors, in dimensions up to $8$. The classification is based on that of the compact case, mainly given in \cite{BhmKrr}, and a duality procedure \cite{Nkn1}. It includes some infinite families, such as the non-compact analogs of the Aloff-Wallach spaces, and some other examples in dimension $8$. To solve the Einstein equation for homogeneous metrics on these spaces we proceed case by case, studying the isotropy representations, and in many cases the results from \cite{Nkn2} can be applied to conclude that there is no solution. However, in some cases --mostly in higher dimensions-- this is not enough, and a very detailed analysis of the Ricci curvature is carried out. As a by-product of this analysis, a general non-existence result for some cases where $\Sl_2(\RR)$ is one of the simple factors of the transitive group is given in Proposition \ref{Propsl2RxG1}.

One of the reasons why we are not able to extend Theorem \ref{mainnonuni} to dimensions $11$ and higher is that already in dimension $11$, examples such as $( \Sl_2(\CC) \cdot \RR )\ltimes \RR^4$ appear, with $N = \RR^4$ and $\Sl_2(\CC)$ acting non-trivially on it. The homogeneous Einstein equation for such a space reduces to an equation for left-invariant metrics on $\Sl_2(\CC)$ which is even more general than the Einstein equation.

%
%
%
%


The article is organized as follows. In Section \ref{prelimstruct} we state the structure theorems for non-compact homogeneous Einstein spaces, since they will be repeatedly used along the paper, and prove the new refinements metioned above. In Section \ref{semisimple} we prove Theorem \ref{thmsemisimple}, which deals with the semisimple case, and in order to do that we give the classification of non-compact semisimple homogeneous spaces up to dimension $8$. This, together with previously known results, already implies Theorem \ref{main6}. In Section \ref{sectionnonuni} we prove Theorem \ref{mainnonuni}, and then in Section \ref{strong} we focus on the strong Alekseevskii conjecture and complete the proofs of Theorems \ref{main7} and \ref{main8}.

\vs \noindent {\it Acknowledgements.} It is our pleasure to thank Jorge Lauret for fruitful discussions, and Christoph B\"ohm for providing useful comments on a draft version of this article.

Part of this research was carried out while the first author was a visitor at McMaster University. She is very grateful to the Department of Mathematics, the Geometry and Topology group and especially to McKenzie Wang for his kindness and hospitality.

\section{Structure of non-compact homogeneous Einstein spaces}\label{prelimstruct}

In this section we review the most important known facts about the algebraic structure of non-compact homogeneous Einstein spaces, since they will be crucial in the proof of our main results. Here and throughout the rest of the article, all manifolds under consideration are connected and all homogeneous spaces are almost-effective, unless otherwise stated.

\begin{theorem}[\cite{alek,JblPtr}]\label{structure}
Let $(M,g)$ be a simply-connected homogeneous Einstein space with negative scalar curvature. Then, there exists a transitive Lie group of isometries $G$ whose isotropy at some point $p\in M$ is $K$, with the following properties:
\begin{itemize}
    \item[(i)] $G = \left(G_1 A\right) \ltimes N$, where $N$ is a nilpotent normal Lie subgroup, $U = G_1 A$ is a reductive Lie group with center $A = Z(U)$, and $G_1 = [U,U]$ is semisimple without any compact simple factors and contains the isotropy $K$.
    \item[(ii)] The orbits of $U$ and $N$ are orthogonal at $p$.
    \item[(iii)] The induced left-invariant metric $g_N$ on $N$ is a Ricci soliton (i.e.~ $(N, g_N)$ is a \emph{nilsoliton}).
    \item[(iv)] The Ricci curvature of the induced $U$-invariant metric $g_{U/K}$ on $U/K$ is given by
    \begin{equation}\label{eqRicU/K}
        \ricci_{U/K}(Y,Y) = c \cdot g_{U/K}(Y,Y) +  \tr \left(S\left(\theta(Y)\right)^2\right),
    \end{equation}
        for some $c<0$, where $\theta : \ug \to \Der(\ngo)$ is the corresponding infinitesimal action ($\ug = \Lie(U), \ngo = \Lie(N)$), $S(A) = \unm\left(A + A^t\right)$, and the transpose is taken relative to the nilsoliton inner product on $\ngo$.
    \item[(v)] The infinitesimal action $\theta$ and the nilsoliton metric satisfy the following compatibility condition:
    \begin{equation}\label{eqmmtheta}
        \sum_i [\theta(Y_i), \theta(Y_i)^t] = 0,
    \end{equation}
    where the sum is taken over an orthonormal basis for $\ug$\footnote{See Remark \ref{remarks}, \eqref{remarkinfinitesimal} below.}. Moreover, $\theta(Y) = \theta(Y)^t$ for every $Y\in \zg(\ug)$.
\end{itemize}
Conversely, if a simply-connected homogeneous manifold admits a transitive group of isometries $G$ satisfying $(i)-(v)$, then it is Einstein, with negative scalar curvature.
\end{theorem}

\begin{remark}\label{remarks}
\begin{enumerate}[(a)]
    \item\label{remarkinfinitesimal} Conditions (i) and (ii) may also be interpreted at the infinitesimal level, as follows:
    Let $\ggo, \ug, \ngo, \kg$ be the Lie algebras of the groups $G, U, N, K$, respectively. We have that $\ggo = \ug \ltimes_\theta \ngo$, with $\ug$ a reductive subalgebra and $\ngo$ the nilradical of $\ggo$ (the maximal nilpotent ideal). Consider the reductive decomposition $\ggo = \kg \oplus \pg$ for $G/K$, where $\pg$ is the orthogonal complement of $\kg$ relative to the Killing form of $\ggo$. This induces a reductive decomposition $\ug = \kg \oplus \hg$ for the homogeneous space $U/K$, by letting $\hg := \pg \cap \ug$. The $G$-invariant metric $g$ on $G/K$ is thus identified with an $\Ad(K)$-invariant inner product $\ip$ on~ $\pg$, and one has that
    \[
        \langle \hg, \ngo \rangle = 0.
    \]
    For technical reasons, it is sometimes convenient to extend this inner product to an inner product on $\ggo$, which we will also denote $\ip$, by letting $\kg \perp\pg$ and choosing on $\kg$ some $\Ad(K)$-invariant inner product. By doing so, we clearly obtain $\langle \ug, \ngo \rangle = 0$. In fact, in condition (v), by an orthonormal basis of $\ug$ we mean that it is orthonormal with respect to the inner product extended as explained above.
    \vspace{1mm}

    \item\label{remarktheta} $\theta : \ug \to \Der(\ngo)$ is nothing but the adjoint representation of $\ggo$ co-restricted to act on the nilradical, that is,
    \[
        \theta(Y) X = [Y,X] \in \ngo, \qquad X\in \ngo, \,\, Y\in \ug.
    \]
    It was noticed by J.~Lauret that condition \eqref{eqmmtheta} is equivalent to $\theta$ being a zero of the moment map associated with the natural $\Gl(\ngo)$-action on the vector space $\End(\ug,\End(\ngo))$ (see \cite[Appendix]{semialglow} and \cite[$\S 2.1$]{JblPtr} for more details on this fact).

    \vspace{1mm}

    \item The Einstein constant of $(G/K,g)$ and the cosmological constant of the nilsoliton $(N,g_N)$ both coincide with the scalar $c<0$ in condition (iv).

    \vspace{1mm}

    \item According to the construction procedure for expanding algebraic solitons described in \cite[\S 5]{alek}, it is easy to see that given any non-compact Einstein homogeneous space $G/K$, we can always build another one with the same $U/K$ but with abelian nilradical.

    \vspace{1mm}

    \item The simply-connected hypothesis is not necessary for obtaining the results at the infinitesimal level. However, it turns out to be necessary for the converse assertion to hold. In particular, one question which still remains unanswered is whether any homogeneous Einstein space with negative scalar curvature is simply connected. This is known to be true when the universal cover is a solvmanifold, by the results in \cite{AC99, Jab15}.

    \vspace{1mm}

    \item\label{rmkDotti} If $G$ is a unimodular Lie group, then by \cite[Theorem 2]{Dtt88} it must in fact be semisimple, and hence it equals $G_1$. In this case, the only information that Theorem \ref{structure} provides is that $G_1$ has no compact simple factors.

    \vspace{1mm}

    \item On the other hand, if $(M,g)$ admits a non-semisimple transitive group of isometries, it follows from \cite{alek, JblPtr} that the group $G$ in Theorem \ref{structure} may be chosen to be non-unimodular. In this case, the so called ``mean curvature vector'' $H$, implicitly defined by
    \[
        \langle H, X \rangle = \tr \left( \ad X\right), \qquad \forall \, X\in\ggo,
    \]
    is non-zero.
\end{enumerate}
\end{remark}
By using that under the hypothesis of Theorem \ref{structure}, $G/K$ is diffeomorphic to the product manifold $G_1/K \times AN$, with $S = AN$ a simply-connected solvable Lie group, one obtains the following
\begin{corollary}\label{reductionG1/K}
Let $(M,g)$ be a simply-connected homogeneous Einstein space with negative scalar curvature, and let $G/K$ be the presentation given in Theorem \ref{structure}. Then, $M$ is diffeomorphic to a Euclidean space if and only if $G_1/K$ is so.
\end{corollary}

It is important to notice that Theorem \ref{structure} does not state that the orbits of $Z(U)$ and $G_1$ are orthogonal at $p$. In other words, it is not known whether $\zg(\ug) \perp \ggo_1$ (where $\ggo_1= \Lie(G_1) = [\ug,\ug]$). This would be the most natural result to expect, since it would imply that there is a Levi decomposition $G = G_1 \ltimes S$ which is adapted to the geometry of $(M^n,g)$, in the sense that the orbits of $G_1$ and $S$ at $p$ are orthogonal. In what follows we prove that in fact one always has this nicer structure.   





\begin{theorem}\label{thm_lemadimn}
Let $(M,g)$ be a simply-connected homogeneous Einstein space of negative scalar curvature, and consider for it the presentation $G/K$ given in Theorem \ref{structure}. Then, $\zg(\ug)$ is orthogonal to $\ggo_1$. 
\end{theorem}

\begin{remark}\label{rmkn1}
If furthermore one has that $\theta|_{\ggo_1} = 0$, then $G/K$ is isometric to a Riemannian product $G_1/K \times S$ of Einstein homogeneous spaces of negative scalar curvature. Notice that condition $\theta|_{\ggo_1} = 0$ is trivially satisfied when $\dim \ngo = 1$.
\end{remark}

\begin{proof}
Following the notation from Remark \ref{remarks}, \eqref{remarkinfinitesimal}, equation \eqref{eqRicU/K} may be rewritten as an equation for endomorphisms of $\hg \simeq T_{eK} U/K$ as
\begin{equation}\label{eqn_RicUK}
	\Ricci_{U/K} = c \cdot I + C_\theta.
\end{equation}
Here, $\Ricci_{U/K} \in \End(\hg)$ denotes the Ricci operator of the homogeneous space $(U/K, g_{U/K})$, and $C_\theta \in \End(\hg)$ is the symmetric endomorphism given by
\[
\langle C_\theta X, Y \rangle = \tr S(\theta(X))S(\theta(Y)), \qquad X,Y\in \hg.
\]
Since $\theta$ is defined on $\ug$ and not just on $\hg$, we may of course extend $C_\theta$ to a symmetric endomorphism of $\ug$, where $C_\theta (\kg) = 0$ (recall that the action of the isotropy is by skew-symmetric operators).

We have $\theta: \ug \to \End(\ngo)$, and by Theorem \ref{structure} $S(\theta(\zg(\ug)))$ is a family of pairwise commuting, symmetric operators in $\End(\ngo)$, which commute also with all of $\theta(\ug)$ (recall that $\theta$ is a Lie algebra representation). We may thus consider an orthogonal decomposition of $\ngo$ into common eigenspaces for the family $S(\theta(\zg(\ug)))$ (i.e.~ a weight-space decomposition):
\begin{equation}\label{rootdec}
\ngo =  \ngo_1 \oplus \ldots \oplus \ngo_l,
\end{equation}
with $\alpha_1, \ldots, \alpha_l \in \zg(\ug)^*$ the corresponding weights. The restricted representation $\theta_{\ggo_1} = \theta|_{\ggo_1}: \ggo_1 \to \End(\ngo)$ must preserve this weight-space decomposition. For each $k = 1,\ldots,l$ we haveß a \emph{co-restricted} representation of the semisimple Lie algbera $\ggo_1$, given by
\[
\theta_{\ggo_1}^k := \pi_k \circ \theta |_{\ggo_1} : \ggo_1 \to \End(\ngo_k),
\]
where $\pi_k:\ngo \to \ngo_k$ is the orthogonal projection. Observe that, in particular, $\theta_{\ggo_1}^k(Y)$ is traceless for each $Y\in \ggo_1$ and $k=1,\ldots,l$.

Now we claim that for $Y\in \ggo_1$, $X\in \zg(\ug)$ one has that $\la C_\theta X, Y\ra = 0$. Indeed, using the orthogonality of the decomposition \eqref{rootdec}, and the fact that it is preserved by $\theta(\ug)$, we obtain
\begin{align*}
    \left\langle C_\theta \, Y, X \right\rangle &= \sum_{k=1}^l \tr \left( S\left(\theta_{\ggo_1}^k (Y)\right) \left(\alpha_k(X) \cdot I\right)\right)  = \sum_{k=1}^l \alpha_k(X) \, \tr \theta_{\ggo_1}^k(Y) = 0.
\end{align*}


Consider in $U/K$ the reductive decomposition $\ug = \kg \oplus \hg$. We may also assume that $\ggo_1 = \kg \oplus \hg_1$ is a reductive decomposition for $G_1/K$, where $\hg_1 \subseteq \hg$. Let $\qg$ be the orthogonal complement of $\hg_1$ in $\hg$, and let us show that $\qg = \zg(\ug)$. To that end, take $Y \in \qg$ and write it as $Y = Y_1 + Y_\zg$, where $Y_1 \in \hg_1$, $Y_\zg \in \zg(\ug)$. We now look at the Ricci curvature in the directions $Y_1$, $Y_\zg$. First, by~ \eqref{eqn_RicUK} and the above claim we obtain
\[
    \Ricci_{U/K}(Y_1, Y_\zg) = c \, \langle Y_1, Y_\zg\rangle = c\, \langle Y_1, Y-Y_1\rangle = - c\, \| Y_1\|^2 \geq 0,
\]
since $c<0$. On the other hand, we use that $Y_\zg \in \zg(\ug)$, $Y\perp [\ug,\ug]$, and the explicit formula for the Ricci curvature in the unimodular case (see \cite[7.38]{Bss}), to get
\begin{align*}
    \Ricci_{U/K}(Y_1, Y_\zg) =& -\unm\sum_{i,j}\langle [Y_1, X_i]_{\hg}, X_j \rangle \langle [Y_\zg, X_i]_{\hg}, X_j \rangle  \\
            & + \unc \sum_{i,j} \langle [X_i,X_j]_{\hg}, Y_1\rangle \langle [X_i,X_j]_{\hg}, Y_\zg\rangle - \unm \tr \ad_\ug Y_1 \ad_\ug Y_\zg \\
            = &\, \unc \sum_{i,j} \langle [X_i,X_j]_{\hg}, Y_1\rangle \langle [X_i,X_j]_{\hg}, Y - Y_1\rangle  \\
            = & -\unc \sum_{i,j} \langle [X_i,X_j]_{\hg}, Y_1\rangle^2 \leq 0,
\end{align*} where $\{ X_i\}$ is an orthonormal basis for $\hg.$
Hence, we must have equality, and $Y_1 = 0$. Therefore, $\qg = \zg(\ug)$, and the proof is finished.
\end{proof}

\begin{remark}\label{rmk_centerorthogonal}
The previous theorem holds more generally for expanding homogeneous Ricci solitons. More precisely, if $(M^n, g)$ is an expanding (i.e. $c<0$) homogeneous Ricci soliton, and $G$ is the full isometry group, then by \cite{Jbl} the soliton is \emph{semi-algebraic}. Therefore by \cite{alek} the homogeneous space $G/K$ satisfies all the nice properties stated in Theorem \ref{structure}, but possibly without the additional conditions proven in \cite{JblPtr} for Einstein spaces (namely, $G_1$ might have compact simple factors, and the action of $\zg(\ug)$ on $\ngo$ might not be by symmetric endomorphisms). Nevertheless, Lemma 3.5 from \cite{JblPtr} still assures that by the compatibility condition \eqref{eqmmtheta} one has that the family $\theta(\zg(\ug)) \subset \End(\ngo)$ consists of normal operators, whose transposes commute with all of $\theta(\ug)$. Thus, one can also consider the decomposition \eqref{rootdec} as in the proof of Theorem~ \ref{thm_lemadimn}, and proceed in exactly the same way to conclude that $\zg(\ug) \perp \ggo_1$.
\end{remark}

 As a quick application we get an alternative proof of the following result of Jablonski \cite{Jbl13b}.

\begin{corollary}\label{cor_algebraic}
Homogeneous Ricci solitons are algebraic.
\end{corollary}

\begin{proof}
As is well-known, the only non-trivial examples (that is, not locally isometric to the product of an Einstein homogeneous space and a flat factor $\RR^k$) occur in the expanding case (see the discussion in \cite[\S 2]{solvsolitons} and the references therein for more details). Let $(M,g)$ be an expanding homogeneous Ricci soliton. For the presentation $G/K$, where $G$ is the full isometry group, we have by Theorem \ref{thm_lemadimn} and Remark \ref{rmk_centerorthogonal} that $\zg(\ug) \perp \ggo_1$. Now recall that the mean curvature vector $H \in \hg \subset \ug$ is always orthogonal to $\ggo_1 = [\ug,\ug]$, since any representation of a semi-simple Lie algebra consists of traceless endomorphisms. Thus, $H \in \zg(\ug)$, and in particular 
\[
	S(\ad H|_\hg) = 0.
\]
By applying Proposition 4.14 from \cite{alek} we conclude that the soliton is indeed algebraic.
\end{proof}

Another application of our new structural results is the reduction of the classification problem (in the non-unimodular case) to the so called ``rank one'' case (cf.~ \cite[Theorem D]{Heb}). 

\begin{corollary}\label{cor_rankone}
Let $(M^n, g)$, $G/K$ be as in Theorem \ref{thm_lemadimn}, with $G$ non-unimodular. Consider $U_0$, $G_0$ the connected Lie subgroups of $U$, $G$ with Lie algebras $\ug_0 := [\ug,\ug]\oplus \RR H \subset \ug$, $\ggo_0 := \ug_0 \oplus \ngo$, respectively. Then, there is a diffeomorphism
\[
	M^n \simeq \RR^a \times G_0 / K, \qquad a = \dim Z(U) - 1,
\]
and the induced $G_0$-invariant metric on $G_0/K$ is Einstein with the same Einstein constant $c<0$ as $g$.
\end{corollary}

\begin{proof}
Recall the following formula for the Ricci curvature of a homogeneous space, whose proof follows immediately from the proof of Proposition 6.1 in \cite{alek}.
\begin{lemma}
Let $(G/K,g)$ be a Riemannian homogeneous space with reductive decomposition $\ggo = \kg \oplus \pg$, and assume there exists $X\in \pg$ such that $[H,X] = 0$, and the subspace $\tilde\ggo := \{X\}^\perp$ is a codimension-one ideal of $\ggo$ that contains $H$ and $\kg$. Let $\widetilde G$ be the connected Lie subgroup of $G$ with Lie algebra $\tilde\ggo$, and consider the induced metric on the orbit $\widetilde G \cdot (eK) \simeq \widetilde G / K$. Then, the corresponding Ricci operators satisfy
\[
    \ricci_{G/K} |_{ \, \, \tilde \pg} = \ricci_{\widetilde G / K} + \unm \left[A,A^t\right],
\]
where $\tilde \pg = \pg \cap \tilde\ggo$ and $A := \ad X|_{\tilde \pg} \in \End(\pg)$.
\end{lemma}

Theorem \ref{thm_lemadimn} implies that $H\in \zg(\ug)$, and that any $X \in \zg(\ug)$ with $X \perp H$ satisfies the conditions of the above lemma. Moreover, the corresponding endomorphism $A$ is symmetric by Theorem~ \ref{structure},~ (v), thus the term $\unm [A,A^t]$ in the formula vanishes. By applying the lemma to any such $X$ we obtain a codimension-one submanifold $\tilde G / K$ in $G/K$ which with the induced metric is Einstein, with the same Einstein constant as $G/K$. Since the spaces are simply-connected, as differentiable manifolds we have that $G/K \simeq \RR \times \tilde G/ K$. After applying this procedure $a$ times, where $a = \dim Z(U)-1$, the corollary follows.
\end{proof}

To conclude this section we prove the following simple but useful formula for the Ricci curvature of a homogeneous space, which is in some way a generalization of \cite[Lemma 2.3]{Mln}.

\begin{lemma}\label{lem_formulaRicci}
Let $(U/K,g)$ be a Riemannian homogeneous space with $U$ a unimodular Lie group, and consider a reductive decomposition $\ug = \kg \oplus \mg$. If $X, Y\in \mg$ are such that $[\kg,X]= [\kg,Y] = 0$, then
\[
     \Ricci(X,Y) =  \unc \sum_{i,j}  \langle [X_i,X_j]_\mg, X\rangle \langle [X_i,X_j]_\mg, Y\rangle - \tr S(\ad_\mg X) S(\ad_\mg Y),
\]
where $\{ X_i\}$ is any orthonormal basis for $\mg$ (here, $\ad_\mg X \in \End(\mg)$ stands for the restriction of $\ad X$ to $\mg$, projected onto $\mg$). Moreover, if $Y$ is orthogonal to the commutator ideal $[\ug,\ug]$ (i.e.~ to its projection onto $\mg$), then
\[
    \Ricci(X,Y) = - \tr S(\ad_\mg X) S(\ad_\mg Y), \qquad \forall \, X\in \mg \mbox{ such that } [\kg,X]=0.
\]
\end{lemma}

\begin{proof}
From the formula \cite[7.38]{Bss} for the Ricci curvature of a homogeneous space, and using that $\ug$ is unimodular, we see that
\begin{align*}
    \Ricci(X,Y)  =& -\unm \sum_{i,j} \langle [X,X_i]_\mg, X_j \rangle \langle [Y,X_i]_\mg, X_j \rangle  \\
                  & + \unc \sum_{i,j}  \langle [X_i,X_j]_\mg, X\rangle \langle [X_i,X_j]_\mg, Y\rangle - \unm B(X,Y) \\
                 =& -\unm \tr \left(\ad_\mg X\right) \left(\ad_\mg Y \right)^t - \unm \tr (\ad X)(\ad Y) \\
                  & + \unc \sum_{i,j}  \langle [X_i,X_j]_\mg, X\rangle \langle [X_i,X_j]_\mg, Y\rangle,
\end{align*}
where $\{ X_i\}$ is an orthonormal basis for $\mg$. Notice that conditions $[\kg,X] = 0$ and $[\kg,Y] =~ 0$ imply that $\tr(\ad X)(\ad Y) = \tr(\ad_\mg X)(\ad_\mg Y)$. Then, the first formula follows. If moreover $Y\perp [\ug,\ug]_\mg$, then it is easy to see that $[\kg,Y]=0$, so the first formula applies, and the sum term in it disappears.
\end{proof}

\section{Semisimple transitive group}\label{semisimple}

The main purpose of this section is to prove the following
\begin{theorem}\label{thmsemisimple}
Let $G$ be a semisimple Lie group and consider a homogeneous Einstein space $\left(G/H,g\right)$ which is de Rham irreducible. Assume that $\dim G/H \leq 8$, $\dim H \geq 1$ and that $G/H \neq$ $(\Sl_2(\RR) \times \Sl_2(\CC))\slash \Delta \U(1)$. Then, $(G/H,g)$ is an irreducible symmetric space of the non-compact type.
\end{theorem}

The proof will follow from a case-by-case analysis. We warn the reader that, in contrast with the rest of the article, throughout this section the group $G$ will always be a semisimple Lie group.

\begin{definition}\label{defsshomogspace}
We call a homogeneous space $G/H$ \emph{semisimple of the non-compact type} if $G$ is a semisimple Lie group without compact simple factors.
\end{definition}

In view of Theorem \ref{structure}, we are reduced to studying the cases where $G/H$ is semisimple of the non-compact type. Moreover, we may restrict ourselves to the simply-connected case. Indeed, the universal cover of an Einstein manifold is still Einstein, and it is a classical result that symmetric spaces of the non-compact type do not admit non-trivial homogeneous quotients~ \cite{Car27}.

Following \cite{Al11, Nkn1}, we use the duality between compact and non-compact symmetric spaces to obtain the classification of semisimple homogeneous spaces of the non-compact type from the classification of compact homogeneous spaces in low dimensions ([BK]), as follows:

Given $G/H$ a semisimple homogeneous space of the non-compact type, let $\ggo = \ggo_1 \oplus \ldots \oplus \ggo_s$ be the decomposition of $\ggo$ into simple ideals --which are all of the non-compact type-- let $\kg\subseteq \ggo$ be a maximal compactly embedded subalgebra such that $\hg \subseteq \kg$, and for each $i=1,\ldots,s$ let $\kg_i = \ggo_i \cap \kg$, which is a maximal compactly embedded subalgebra of $\ggo_i$. The pairs $(\ggo_i, \kg_i)$ are symmetric pairs of the non-compact type (at the Lie algebra level), and its corresponding dual symmetric pairs $(\hat\ggo_i, \hat\kg_i)$ are of the compact type. If $\hat \ggo := \hat \ggo_1 \oplus \ldots\oplus \hat\ggo_s$, $\hat \kg := \hat\kg_1 \oplus \ldots \oplus \hat\kg_s$, then $\hat \kg$ and $\kg$ are isomorphic Lie algebras, and via this isomorphism we can consider the subalgebra $\hat \hg \subseteq \hat \kg$ corresponding to $\hg \subseteq \kg$. The effective homogeneous space $\hat G/ \hat H$ associated with $\hat \ggo, \hat\hg$ is compact.

Therefore, in order to obtain all possible spaces $G/H$ as above one can argue as follows:
\begin{itemize}
    \item Classify all compact homogeneous spaces in ``canonical presentation'' (in the sense of \cite{BhmKrr}).
    \item For each compact homogeneous space $(\hat G/ \hat H)$ in the previous classification, consider all possible compact symmetric pairs $(\hat \ggo,\hat\kg)$ with $\hat \hg \subseteq \hat \kg$, where $\Lie(\hat G) = \hat\ggo$, $\Lie(\hat H) = \hat \hg$ (there may be none at all).
    \item For each such pair, let $(\ggo,\kg)$ be its dual, obtained by dualizing each simple factor to its non-compact counterpart. The isomorphism $\kg \simeq \hat \kg$ defines a subalgebra $\hg\subseteq \kg$ isomorphic to $\hat \hg$, and from $\ggo, \hg$ one obtains a non-compact homogeneous space $G/H$ as desired.
\end{itemize}

We note that if a non-compact $G/H$ is obtained from a compact $\hat G/ \hat H$, then the Lie groups $H$ and $\hat H$ are isomorphic, and moreover the isotropy representations are equivalent.
\begin{remark}
To obtain the classification of all non-compact homogeneous spaces with semisimple transitive group (i.e.~ taking into account that $G$ may have compact simple factors), the duality procedure works in the very same way. One only needs to dualize the symmetric pairs which are of the non-compact type.
\end{remark}

We give in Table \ref{tabla} the classification of simply-connected, semisimple homogeneous spaces of the non-compact type (cf.~ Definition \ref{defsshomogspace}), together with its corresponding compact duals, the compact symmetric space used in each case for the dualization procedure, and the decomposition of the isotropy representation into irreducible summands. Notice also that, for notational purposes, some of the non-compact spaces in the table are not simply connected, but still they are to be read as their universal covers. Symmetric spaces are not included, since a list of all irreducible symmetric spaces can be found for instance in \cite[p.~200]{Bss}. We also do not include cases which are product of lower dimensional homogeneous spaces, unless the space admits non-product invariant metrics (see Proposition \ref{prodRiem} below). Our notation follows that of \cite{Bss}, with the only exception of $\SU(1,1)$, which we call $\Sl_2(\RR)$.

\afterpage{
    \clearpage
    \thispagestyle{empty}
    \begin{landscape}
        \scriptsize
        \centering
        \begin{tabular}{|>{\hspace{-3pt}}c<{\hspace{-3pt}}|c|c|c|c|>{\hspace{-3pt}}c<{\hspace{-3pt}}|}
          \hline
          $\dim$                & $\hat G/ \hat H$ (compact)                                                        & $\hat G/ \hat K$ (symmetric)                                              & $G/H$ (noncompact)                                                                        & Isotropy representation                                                                           & Note  \\
          \hline
          \multirow{1}{*}{$3$}  & $\SU(2)/\{\hbox{id}\}$                                                            & $\SU(2)/\U(1)$                                                            & $\Sl_2(\RR)/\{\hbox{id}\}$                                                                & Lie group                                                                                         &       \\ \cline{2-6}
          \hline
          \multirow{4}{*}{$5$}  & \multirow{3}{*}{$(\SU(2)\times\SU(2))/\Delta_{p,q}\U(1)$}                         & $(\SU(2){\times} \SU(2))/\Delta \SU(2)$                                   & $\Sl_2(\CC)/\U(1)$                                                                        & $\qg_1^{(2)}\oplus \pg_0^{(1)}\oplus \pg_1^{(2)}, \, \qg_1 \simeq \pg_1$                          &       \\ \cline{3-6}
                                &                                                                                   & \multirow{2}{*}{$\SU(2)/\U(1) {\times} \SU(2)/\U(1)$}                     & \multirow{2}{*}{$(\Sl_2(\RR)\times \Sl_2(\RR))/\Delta_{p,q} \SO(2)$}                      & $\qg_0^{(1)}\oplus\pg_1^{(2)} \oplus \pg_2^{(2)},$                                                &  \multirow{2}{*}{\ref{S2xS3}}     \\
                                &                                                                                   &                                                                           &                                                                                           & $\pg_1\simeq \pg_2 \iff p=q $                                                                     &       \\ \cline{2-6}
                                & $\SU(3)/\SU(2)$                                                                   & $\SU(3)/\U(2)$                                                            &  $\SU(2,1)/\SU(2)$                                                                        & $\qg_0^{(1)}\oplus\pg_1^{(4)}$                                                                    &       \\ \cline{2-6}
          \hline
          \multirow{5}{*}{$6$}  & \multirow{2}{*}{$(\SU(2) \times \SU(2))/\{\hbox{id}\}$}                           & $(\SU(2)\times \SU(2))/\Delta \SU(2)$                                     & $\Sl_2(\CC)/\{\hbox{id}\}$                                                                & Lie group                                                                                         &       \\ \cline{3-6}
                                &                                                                                   & $\left(\SU(2)/\U(1)\right)^2$                                             & $(\Sl_2(\RR)\times \Sl_2(\RR))/\{\hbox{id}\}$                                             & Lie group                                                                                         &       \\ \cline{2-6}
                                & $\Spe(2)/\Spe(1) \U(1)$                                                           & $\Spe(2)/\Spe(1) \Spe(1)$                                                 & $\Spe(1,1)/\Spe(1) \U(1)$                                                                 & $\qg_1^{(2)} \oplus \pg_1^{(4)}$                                                                  &       \\ \cline{2-6}
                                & $G_2/\SU(3)$                                                                      & -                                                                         & -                                                                                         & Irreducible                                                                                       &  \ref{isotropyirred}     \\ \cline{2-6}
                                & $\SU(3)/T^2$                                                                      & $\SU(3)/\U(2)$                                                            & $\SU(2,1)/T^2$                                                                            &  $\qg_1^{(2)} \oplus \pg_1^{(2)} \oplus \pg_2^{(2)}$                                              &       \\ \cline{2-6}
          \hline
          \multirow{14}{*}{$7$} & $\Spein(7)/G_2$                                                                   & -                                                                         & -                                                                                         & Irreducible                                                                                       &  \ref{isotropyirred}     \\ \cline{2-6}
                                & $\Spe(2)/\SU(2)$                                                                  & -                                                                         & -                                                                                         & Irreducible                                                                                       &  \ref{isotropyirred}     \\ \cline{2-6}
                                & $\SU(4)/\SU(3)$                                                                   & $\SU(4)/\U(3)$                                                            & $\SU(3,1)/\SU(3)$                                                                         & $\qg_0^{(1)}\oplus\pg_1^{(6)}$                                                                    &       \\ \cline{2-6}
                                & $\Spe(2)/\Spe(1)$                                                                 & $\Spe(2)/\Spe(1)\Spe(1)$                                                  & $\Spe(1,1)/\Spe(1)$                                                                       & $\qg_0^{(3)}\oplus\pg_1^{(4)}$                                                                    &       \\ \cline{2-6}
                                & \multirow{2}{*}{$\SO(5)/\SO(3)$}                                                  & $\SO(5)/\SO(4)$                                                           & $\SO(4,1)/\SO(3)$                                                                         & $\qg_1^{(3)}\oplus \pg_0^{(1)}\oplus \pg_1^{(3)}$, $\qg_1 \simeq \pg_1$                           &       \\ \cline{3-6}
                                &                                                                                   & $\SO(5)/\SO(3)\SO(2)$                                                     & $\SO(3,2)/\SO(3)$                                                                         & $\qg_0^{(1)}\oplus\pg_1^{(3)}\oplus\pg_2^{(3)}$, $\pg_1 \simeq \pg_2$                             &       \\ \cline{2-6}
                                & \multirow{4}{*}{$\SU(3)/\Delta_{p,q}\U(1)$}                                       & \multirow{3}{*}{$\SU(3)/\U(2)$}                                           & \multirow{3}{*}{$\SU(2,1)/\Delta_{p,q}\U(1)$}                                             & $\qg_0^{(1)} \oplus \qg_1^{(2)} \oplus \pg_1^{(2)} \oplus \pg_2^{(2)}$,                           &       \\
                                &                                                                                   &                                                                           &                                                                                           & $\pg_1\simeq \pg_2 \iff p=q=1,$                                                                   &   \ref{Aloff-Wallach}   \\
                                &                                                                                   &                                                                           &                                                                                           & $\qg_1 \simeq \pg_1 \iff p=0, q=1 $                                                               &       \\ \cline{3-6}
                                &                                                                                   & $\SU(3)/\SO(3)$                                                           & $\Sl_3(\RR)/\SO(2)$                                                                       & $\qg_1^{(2)}\oplus\pg_0^{(1)}\oplus\pg_1^{(2)}\oplus\pg_2^{(2)}, \, \qg_1\simeq\pg_1$             &       \\ \cline{2-6}
                                &$(\SU(3){\times}\SU(2))/\Delta_{p,q}\U(1)(\SU(2){\times}\{\hbox{id}\})$            & $\SU(3)/\U(2) {\times} \SU(2)/\U(1)$                                      & $(\SU(2,1){\times}\Sl_2(\RR))/\Delta_{p,q}\U(1)(\SU(2){\times}\{\hbox{id}\})$             & $\qg_0^{(1)} \oplus \pg_1^{(4)} \oplus \pg_2^{(2)} $                                              &   \ref{S2xS3}    \\ \cline{2-6} 
                                & $(\SU(2)\times \SU(2) \times \SU(2))/\Delta_{a,b,c}T^2$                           & $\left(\SU(2)/\U(1)\right)^3$                                             & $(\Sl_2(\RR)\times \Sl_2(\RR) \times \Sl_2(\RR))/\Delta_{a,b,c}T^2$                       & $\qg_0^{(1)} \oplus \pg_1^{(2)}\oplus\pg_2^{(2)}\oplus\pg_3^{(2)}$                                &       \\ 
          \hline
          \multirow{12}{*}{$8$} & \multirow{2}{*}{$\SU(3)/\{\hbox{id}\}$}                                           & $\SU(3)/\U(2)$                                                            & $\SU(2,1)/\{\hbox{id} \}$                                                                 & Lie group                                                                                         &       \\ \cline{3-6}
                                &                                                                                   & $\SU(3)/\SO(3)$                                                           & $\Sl_3(\RR)/\{\hbox{id} \}$                                                               & Lie group                                                                                         &       \\ \cline{2-6}
                                & \multirow{2}{*}{$\Spe(2)/T^2$}                                                    & $\Spe(2)/\U(2)$                                                           & $\Spe(2,\RR)/T^2$                                                                         & $\qg_1^{(2)}\oplus \pg_1^{(2)}\oplus \pg_2^{(2)}\oplus \pg_3^{(2)}$                               &       \\ \cline{3-6}
                                &                                                                                   & $\Spe(2)/\Spe(1)\Spe(1)$                                                  & $\Spe(1,1)/T^2$                                                                           & $\qg_1^{(2)}\oplus \qg_2^{(2)}\oplus \pg_1^{(2)}\oplus \pg_2^{(2)}$                               &       \\ \cline{2-6}
                                & \multirow{4}{*}{$(\SU(2)\times\SU(2)\times\SU(2))/\Delta_{a_1,a_2,a_3}\U(1)$}     & \multirow{2}{*}{$\SU(2)/\U(1){\times}\left(\SU(2)^2/\Delta\SU(2) \right)$}& \multirow{2}{*}{$(\Sl_2(\RR)\times\Sl_2(\CC))/\Delta_{p,q} \U(1)$}                        & $\qg_0^{(1)}\oplus \qg_1^{(2)}\oplus \pg_0^{(1)} \oplus \pg_1^{(2)} \oplus \pg_2^{(2)}, $         &  \multirow{4}{*}{\ref{a1a2a3}}     \\  
                                &                                                                                   &                                                                           &                                                                                           & $  \,\qg_1\simeq\pg_1, \pg_1\simeq \pg_2 \iff p=q$                                                &       \\ \cline{3-5}
                                &                                                                                   & \multirow{2}{*}{$(\SU(2)/\U(1))^3$}                                       & \multirow{2}{*}{$(\Sl_2(\RR)\times\Sl_2(\RR)\times\Sl_2(\RR))/\Delta_{a_1,a_2,a_3}\U(1)$} & $\qg_0^{(2)} \oplus \pg_1^{(2)} \oplus \pg_2^{(2)} \oplus \pg_3^{(2)},$                           &       \\
                                &                                                                                   &                                                                           &                                                                                           & $\pg_i\simeq \pg_j \iff a_i=a_j$                                                                  &       \\ \cline{2-6}
                                & \multirow{3}{*}{$\SU(2)\times(\SU(2)\times \SU(2))/\Delta_{p,q}\U(1)$}            & $\SU(2)/\U(1){\times}\left(\SU(2)^2/\Delta\SU(2) \right)$                 & $\Sl_2(\RR) \times  \Sl_2(\CC)/ \U(1)$                                                    & $\qg_0^{(1)}\oplus \qg_1^{(2)} \oplus \pg_0^{(3)} \oplus \pg_1^{(2)},\, \qg_1 \simeq \pg_1$       &       \\ \cline{3-6}
                                &                                                                                   & \multirow{2}{*}{$(\SU(2)/\U(1))^3$}                                       & \multirow{2}{*}{$\Sl_2(\RR)\times(\Sl_2(\RR)\times\Sl_2(\RR))/\Delta_{p,q}\U(1)$}         & $\qg_0^{(2)} \oplus \pg_0^{(2)} \oplus \pg_1^{(2)}\oplus \pg_2^{(2)}$                             &  \multirow{2}{*}{\ref{S2xS3}}     \\
                                &                                                                                   &                                                                           &                                                                                           & $\pg_1\simeq \pg_2 \iff p = q$                                                                    &       \\ \cline{2-6}
                                & $\SU(2)\times (\SU(3)/\SU(2))$                                                    & $\SU(2)/\U(1) \times \SU(3)/\U(2)$                                        & $\Sl_2(\RR) \times \SU(2,1)/\SU(2)$                                                       & $\qg_0^{(2)}\oplus \pg_0^{(2)} \oplus \pg_1^{(4)} $                                               &       \\ \cline{2-6}
          \hline
        \end{tabular}
        \captionof{table}{Non-symmetric, non-product, non-compact homogeneous spaces with semisimple transitive group without compact simple factors, and its corresponding compact duals, in dimensions less than or equal to $8$.}
        \label{tabla}
    \end{landscape}
    \clearpage
}


Regarding the list of compact homogeneous spaces in canonical presentation, we refer the reader to \cite{BhmKrr}. All the embedings of the isotropy subgroup are clear once the corresponding compact symmetric space used for the dualization is taken into account. The precise meaning of the parameters corresponding to abelian subgroups in the isotropy may be found in \cite[~\S 1]{Nkn04}.

The information on the isotropy representation is to be understood as follows: for a space $G/H$, consider $\hg \subseteq \kg \subseteq \ggo$ as above, where $\kg$ is a maximal compactly embedded subalgebra with corresponding connected subgroup $K$. Take the corresponding Cartan decomposition $\ggo = \kg \oplus \pg$ (cf.~\cite[pp.~182]{Helgason}), and let $\qg$ be an $\Ad(K)$-invariant complement for $\hg$ in $\kg$. Setting $\mg = \qg \oplus \pg$, we obtain a reductive decomposition $\ggo = \hg \oplus \mg$ for the homogeneous space $G/H$. Whenever we write $\sum_i \qg_i^{(a_i)} \oplus \sum_j \pg_j^{(b_j)}$ we mean that
\[
    \qg = \sum_i \qg_i^{(a_i)}, \qquad \pg = \sum_j \pg_j^{(b_j)}, \qquad \dim \qg_i^{(a_i)} = a_i, \quad  \dim \pg_j^{(b_j)} = b_j,
\]
and for $i,j \neq 0$, each summand $\qg_i^{(a_i)}$, $\pg_j^{(b_j)}$ is an irreducible $\Ad(H)$-module, where any two such modules are inequivalent unless otherwise stated. The $0$ sub-index stands for trivial modules (i.e.~ $[\hg,\qg_0^{(a_0)}] = [\hg,\pg_0^{(b_0)}] = 0$).

\vspace{2mm}
\noindent \emph{Notes on Table \ref{tabla}:}
\vspace{-1mm}
\begin{enumerate}[(a)]
  \item\label{Aloff-Wallach} $p,q\in \ZZ$, $0\leq p \leq q $, $\operatorname{gcd}(p,q) = 1$. See \cite{Wng82}.
  \item\label{S2xS3} $p,q\in \ZZ-\{0\}$, $p\leq q$, $\operatorname{gcd}(p,q) = 1$.
  \item\label{isotropyirred} These compact spaces are isotropy irreducible but non-symmetric (see \cite[pp.~ 203]{Bss}). Clearly, they do not admit any non-compact counterpart.
  \item\label{a1a2a3} $a_1,a_2,a_3 \in \ZZ-\{0\}$, $a_1\leq a_2 \leq a_3$, $\operatorname{gcd}(a_1,a_2,a_3) = 1$ (the order may be assumed up to equivariant diffeomorphism, by using outer automorphisms given by the Weyl group; the parameters are all nonzero since otherwise the space splits as a product, and this are considered as a separate case). The space $(\Sl_2(\RR) \times \Sl_2(\CC))\slash \Delta_{p,q} \U(1)$ is obtained only when $a_2 = a_3$. For convenience, we have renamed the parameters as $p=a_1$, $q = a_2 = a_3$.
\end{enumerate}

Recall that by \cite[Theorem 1]{Nkn2}, if a $G$-invariant metric makes the chosen Cartan decomposition orthogonal (i.e.~it is such that $\langle \qg, \pg \rangle = 0$), then $(G/H,g)$ is not Einstein. In particular, if we consider a decomposition of $\mg$ into irreducible $\Ad(H)$-modules given by $\qg = \qg_1 \oplus \ldots \oplus \qg_u$, $\pg = \pg_1 \oplus \ldots \oplus \pg_v$ (recall that $\qg$ and $\pg$ are $\Ad(H)$-invariant), and none of the $\qg_i$ is equivalent to any of the $\pg_j$, then every $G$-invariant metric on $G/H$ satisfies $\langle \qg,\pg\rangle=0$, and thus none of them is Einstein \cite[Corollary]{Nkn2}. It may be the case that a single Cartan decomposition is not orthogonal with respect to \emph{every} $G$-invariant metric, and still every metric makes \emph{some} Cartan decomposition orthogonal (recall that a Cartan decomposition is only unique up to the action of inner automorphisms). In \cite[Theorem~ 2]{Nkn2}, necessary and sufficient conditions are given for this to happen.

The following result is well-known, but we include a proof of it for the sake of completeness.

\begin{proposition}\label{prodRiem}
Let $G_1\slash H_1$, $G_2\slash H_2$ be two homogeneous spaces such that the isotropy representation of $G_1/H_1$ acts non-trivially on every invariant subspace. Then, any $\left(G_1 \times G_2\right)$-invariant metric on $\left(G_1 \times G_2\right)\slash \left(H_1 \times H_2\right)$ is a Riemannian product of invariant metrics on each factor.
\end{proposition}
\begin{proof}
If $\ggo_1=\hg_1\oplus\mg_1$ and $\ggo_2=\hg_2\oplus\mg_2$ are reductive decompositions of $G_1/H_1$ and $G_2/H_2$ respectively, then $\ggo_1\oplus\ggo_2=(\hg_1\oplus\hg_2)\oplus(\mg_1\oplus\mg_2)$ is a reductive decomposition of $G_1 \times G_2/H_1 \times H_2.$ Let $\pg_i\subseteq \mg_i$ be $\ad(\hg_i)$-irreducible subspaces, $i=1,2$. We know that $\ad(\hg_1)|_{\pg_1}$ is non-trivial. If there was an intertwining operator $T:\pg_1\rightarrow\pg_2$, i.e,
\[
    T\circ\ad(Z)|_{\pg_1}=\ad(Z)|_{\pg_2}\circ T, \quad \mbox{for all } Z\in\hg_1\oplus\hg_2,
\]
we could take $Z=(Z_1,0) \in \hg_1\oplus\hg_2$ and would have that $T\circ\ad(Z_1)|_{\pg_1}=0,$ for all  $Z_1 \in \hg_1,$ so $\ad(Z_1)|_{\pg_1}=0$, for all $Z_1 \in \hg_1,$ which is a contradiction.
\end{proof}

We are now in a position to start the case-by-case analysis.

\subsection{$\dim G/H \leq 7$}$ $

After having computed all the isotropy representations, we see that in most of the spaces of dimension up to $7$ in Table \ref{tabla}, the Cartan decomposition we have chosen is such that $\qg$ and $\pg$ share no equivalent modules, and thus these spaces admit no $G$-invariant Einstein metric. The exceptions are the following:
\[
    \Sl_2(\CC)/\U(1), \quad \SO(4,1)/\SO(3), \quad \SU(2,1)/\Delta_{p,q}\U(1), \quad \Sl_3(\RR)/\SO(2).
\]
Non-existence of homogeneous Einstein metrics on $\Sl_3(\RR)/\SO(2)$ was established in \cite[~Example 4]{Nkn2} by finding, for every $G$-invariant metric, a suitable Cartan decomposition which is orthogonal. By applying the same methods and a straightforward computation, it can be shown that the space $\SO(4,1)/\SO(3)$ also satisfies the hypotheses of \cite[Theorem 2]{Nkn2}, and hence it admits no homogeneous Einstein metric.

\subsubsection{$\Sl_2(\CC)/\U(1)$}\label{sectionsl2C}

Unfortunately, this space admits invariant metrics for which there is no orthogonal Cartan decomposition.

Consider the following ordered basis $\mathcal{B}$ for $\slg_2(\CC)$
\begin{align}\label{matricessl2C}
    Z &= \twomatrix{i}{0}{0}{-i}, \quad Y_0 = \twomatrix{1}{0}{0}{-1}, \quad Y_1 = \twomatrix{0}{1}{1}{0},\\
    Y_2 &= \twomatrix{0}{i}{-i}{0}, \quad X_1 =  \twomatrix{0}{1}{-1}{0}, \quad X_2 = \twomatrix{0}{i}{i}{0}.\nonumber
\end{align}
The isotropy subalgebra is given by $\hg = \RR Z$, $\mg = \operatorname{span}_\RR \{ Y_0, Y_1, Y_2, X_1, X_2\}$ is a reductive complement, and it decomposes into irreducible modules as $\mg = \pg_0 \oplus \pg_1 \oplus \qg_1$, where $\pg_0 = \RR Y_0$, $\pg_1 = \RR Y_1 \oplus \RR Y_2$, $\qg_1 = \RR X_1 \oplus \RR X_2$. Also, if $\kg = \hg \oplus \qg_1 \simeq \sug(2)$, $\pg = \pg_0 \oplus \pg_1$, then $\slg_2(\CC) = \kg \oplus \pg$ is a Cartan decomposition. Let us fix an inner product $\ip_B$ on $\slg_2(\CC)$ that makes $\mathcal{B}$ orthonormal (this inner product is, up to a scalar multiple, the one given by the Killing form of $\slg_2(\CC)$, after reversing its sign on the subalgebra~ $\kg$). Finally, let $\ip_0 = \ip_B \big|_{\mg \times \mg}$, which is of course $\Ad(\U(1))$-invariant.

\begin{lemma}\label{lemmasl2C}
Up to isometry, $\Sl_2(\CC)$-invariant metrics on $\Sl_2(\CC)\slash\U(1)$ can be parameterized by $\Ad(\U(1))$-invariant inner products on $\mg$ of the form
\[
    \langle \cdot, \cdot \rangle_h = \langle h\, \cdot\, , h\, \cdot \rangle_0,
\]
where $h\in \Gl_5(\RR)$ is given by
\[
    h = \left[\begin{array}{ccccc} e &0 &0 &0 & 0\\ 0& a &0 &0 &0 \\0 & 0& a & 0&0 \\0 &0 & -d& b &0 \\ 0&d &0 &0 & b \end{array}\right], \qquad a,b,d,e \in \RR, \quad a,b,e\neq 0.
\]
Moreover, for each such metric, the Ricci curvature satisfies
\[
    \Ricci(h^{-1 }Y_1, h^{-1} X_2) = 4 \, d \cdot \left( (a^2 - e^2)^2 + a^2(b^2 + d^2) \right)a^{-3} b^{-2} e^{-2}.
\]
\end{lemma}

\begin{proof}
Since the modules $\pg_1$ and $\qg_1$ are the only equivalent modules, and they are of complex type, it is clear that the metrics are parameterized by inner products $\ip_h$ on $\mg$, where $h$ is as in the statement, but with a $2\times 2$ block of the form $\minimatrix{c}{-d}{d}{c}$ mapping $\pg_1$ to $\qg_1$. Using that $\ip_h$ and $\ip_{h \cdot T}$ give rise to isometric metrics for any $T = \Ad\left(\exp{t Y_0}\right) \in \Aut(\slg_2(\CC))$, it is easy to find $t$ so that the matrix $h \cdot T$ has the desired form.

The formula for the Ricci curvature follows from a routine (though somewhat lengthy) computation.
\end{proof}

The importance of the previous formula for the Ricci curvature is that this off-diagonal entry vanishes if and only if $d=0$ (recall also that $\langle h^{-1} Y_1, h^{-1} X_2 \rangle_h = 0$). But if $d=0$ then the Cartan decomposition is orthogonal, and the metric is non-Einstein.

\subsubsection{$\SU(2,1)/\Delta_{p,q}\U(1)$} These spaces are the non-compact analogous of the well-known Aloff-~Wallach spaces \cite{AW75}. As long as $p\neq 0$, the Cartan decomposition is orthogonal with respect to any $\SU(2,1)$-invariant metric, hence none of them is Einstein by \cite{Nkn2}. However, the space corresponding to $p=0$, $q=1$ admits $\SU(2,1)$-invariant metrics which make no Cartan decomposition orthogonal. Let us have a closer look at the Lie algebra $\sug(2,1)$: an $\Ad(\Delta_{0,1}\U(1))$-invariant decomposition is given by $\sug(2,1) = \hg_{0,1} \oplus \qg_0 \oplus \qg_1 \oplus \pg_1 \oplus \pg_2$, where
\begin{align*}
    \hg_{0,1} &=  \RR \threematrix{0}{0}{0}{0}{i}{0}{0}{0}{-i}, \quad \qg_0 =  \RR \threematrix{2i}{0}{0}{0}{-i}{0}{0}{0}{-i},
    \quad \qg_1 = \left\{ \threematrix{0}{z}{0}{-\bar{z}}{0}{0}{0}{0}{0} : \, z\in \CC  \right\},  \\
    & \pg_1 = \left\{ \threematrix{0}{0}{z_1}{0}{0}{0}{\bar{z_1}}{0}{0} : \, z_1\in \CC  \right\}, \quad
    \pg_2 = \left\{ \threematrix{0}{0}{0}{0}{0}{z_2}{0}{\bar{z_2}}{0} : \, z_2\in \CC  \right\},
\end{align*}
and the modules $\qg_1$ and $\pg_1$ are equivalent. Any invariant metric would then make the subspaces $\qg_0$, $\pg_2$ and $\qg_1\oplus \pg_1$ orthogonal. But observe that $\ad(\qg_0)$ acts trivially on $\qg_0$ and $\pg_2$, and it acts precisely as the isotropy $\hg_{0,1}$ on $\qg_1\oplus \pg_1$. This immediately implies that for any invariant metric, $\ad(\qg_0)$ consists of skew-symmetric endomorphisms, and hence by Lemma \ref{lem_formulaRicci} the Ricci curvature is non-negative in this directions. Therefore, $\SU(2,1)\slash \Delta_{0,1}\U(1)$ admits no invariant metrics of negative Ricci curvature.

\subsection{$\dim G/H = 8$} $ $

The first two spaces of dimension $8$ in Table \ref{tabla} are Lie groups and will be omitted. The next two cases correspond to homogeneous spaces $G/H$ where $\rank G = \rank H$. As is well known, this implies that the isotropy representation decomposes a sum of pairwise inequivalent modules. Clearly, any Cartan decomposition will be orthogonal with respect to any $G$-invariant metric, and thus none of those can be Einstein by \cite{Nkn2}. For the infinite family of homogeneous spaces $\left(\Sl_2(\RR) \times \Sl_2(\RR) \times \Sl_2(\RR)\right) / \Delta_{a_1,a_2,a_3} \U(1)$ ($6$-th line in the table), the isotropy representation may have some equivalent modules in some special cases, but they are all contained in the subspace $\pg$ of the Cartan decomposition. This implies that for any $G$-invariant metric one still has $\qg\perp\pg$, and none of them can be Einstein.

Let us now consider the following spaces:
\[
\Sl_2(\RR)\times \Sl_2(\CC)/\U(1),\quad \Sl_2(\RR) \times \left(\Sl_2(\RR)\times \Sl_2(\RR)\right)/\Delta_{p,q} \U(1), \quad \Sl_2(\RR) \times \SU(2,1)/\SU(2).
\]
They are all of the form $\Sl_2(\RR) \times G_1 / H$, for some semisimple Lie group $G_1$. Notice that all of them admit metrics which are not Cartan-orthogonal for \emph{any} Cartan decomposition, and hence \cite{Nkn2} can not be applied. Another problem that arises when studying the Einstein equation in these spaces is that whenever the isotropy representation of the space $G_1/H$ has some trivial modules, then the space $G/H$ admits non-product $G$-invariant metrics (cf.~ Proposition \ref{prodRiem}). However, there is still \emph{some} control on such trivial modules. Namely, an easy computations with Lie brackets shows that for the spaces under consideration we have $[\mg_0,\mg_0] \subseteq \hg$, where $\mg_0$ represents the trivial module in $G_1/H$. By looking at the Ricci curvature of $G/H$ at $eH$ in directions tangent to the orbit of $G_1$, and in directions orthogonal to this orbit, we were able to show that if the Ricci curvature preserves this orthogonality, then this forces the metric to be a product (clearly, for an Einstein metric such orthogonality would automatically be preserved by the Ricci curvature). Since $\Sl_2(\RR)$ does not admit any left-invariant Einstein metric, this proves that $G/H$ does not either.


\begin{proposition}\label{Propsl2RxG1}
Let $G/H = \Sl_2(\RR)\times \left( G_1/ H\right)$ be a homogeneous space with $G_1$ semisimple, and assume that $N_{G_1}(H) / H$ is abelian. Then, $G/H$ admits no $G$-invariant Einstein metric.
\end{proposition}

\begin{proof}
Assume that there exists a $G$-invariant Einstein metric $g$ on $G/H$. Let $\ggo_1 = \hg \oplus \mg$ be an $\Ad(H)$-invariant decomposition, and further decompose $\mg = \mg_0 \oplus \mg_1$, where $\hg \oplus \mg_0 = \Lie(N_{G_1}(H))$. Then $\mg_0 = \{X\in \mg : [\hg,X] = 0 \}$ is the trivial $\Ad(H)$-module, $\mg_1$ is the sum of all non-trivial $\Ad(H)$-modules of $\mg$, and our assumption implies that $[\mg_0,\mg_0]\subseteq \hg$. By setting $\pg_0 = \slg_2(\RR) \oplus \mg_0$, $\pg = \pg_0 \oplus \mg_1$, we have a reductive decomposition for $G/H$ given by $\ggo = \hg \oplus \pg$, and $\pg_0$ corresponds to the trivial module. In particular, $\pg_0 \perp \mg_1$. Setting $\lgo := \mg_0^\perp \subseteq \pg_0$, we obtain the orthogonal decomposition
\[
    \pg = \rlap{$\overbrace{\phantom{\lgo\overset{\perp}\oplus\mg_0}}^{\pg_0}$} \lgo \overset{\perp}\oplus \underbrace{\mg_0\overset{\perp}\oplus\mg_1}_\mg.
\]
Our assumption $[\mg_0,\mg_0]\subseteq \hg$ implies that the following bracketing relations are satisfied:
\begin{align}
    [\hg,\pg_0] &= 0,        &   [\hg,\mg_1] &\subseteq \mg_1,       &   [\lgo,\lgo] &\subseteq [\pg_0,\pg_0] \subseteq \hg\oplus\pg_0,  \label{bracketrelations}\\
   [\lgo,\mg_1]&\subseteq \hg\oplus\mg,  &   [\mg_0,\pg_0]&\subseteq \hg,        &   [\mg,\mg_1]&\subseteq \hg\oplus \mg. \nonumber
\end{align}
Since $[\hg,\pg_0]=0$ we may use Lemma \ref{lem_formulaRicci} to obtain
\[
    \langle \Ricci X, Y  \rangle = \unc \sum_{r,s} \langle [U_r, U_s]_\pg, X\rangle \langle [U_r, U_s]_\pg, Y\rangle - \unm \tr S\left( \ad_\pg X\right) S\left( \ad_\pg Y\right),
\]
where $X,Y \in \pg_0$ and $\{U_r\}$ is any orthonormal basis for $\pg$. Assume from now on that $X\in \mg_0$, $Y\in \lgo$, and that $\{ U_r\}$ is the union of orthonormal basis for $\lgo, \mg_0$ and $\mg_1$. Noticing that by \eqref{bracketrelations} one has that $[\pg,\mg] \perp \lgo$ and that $\ad_\pg X$ only acts nontrivially on $\mg$, the above formula simplifies as
\[
    \langle \Ricci X, Y  \rangle = \unc \sum_{U_r,U_s\in \lgo} \langle [U_r, U_s]_\pg, X\rangle \langle [U_r, U_s]_\pg, Y\rangle - \unm \tr S\left( \ad_\mg X\right) S\left( \ad_\mg Y\right).
\]
Choose an orthonormal basis $\{Y_i \}_{i=1}^3$ for $\lgo$, with $Y_i = A_i + B_i$, $A_i\in \slg_2(\RR)$, $B_i\in \mg_0$, and such that $\{A_i \}$ is a \emph{Milnor basis} for $\slg_2(\RR)$, with brackets
\[
   [A_2,A_3] = \alpha A_1, \qquad [A_3,A_1] = \beta A_2, \qquad  [A_1, A_2] = \gamma A_3, \qquad\alpha,\beta,\gamma \neq 0.
\]
Also, choose an orthonormal basis $\{ X_j^0\}_{j=1}^d$ for $\mg_0$ so that $\tr S(\ad_\mg X_i^0) S(\ad_\mg X_j^0) = 0$ if $i\neq j$ (this is indeed possible since the application $(X,Y) \mapsto \tr S(\ad_\mg X)S(\ad_\mg Y)$ is a symmetric bilinear form on $\mg_0$). Then, a straightforward calculation shows that
\[
    \langle \Ricci X_j^0, Y_3 \rangle = -\langle B_3, X_j^0\rangle \left( \gamma^2 + \tr S(\ad_\mg X_j^0)^2 \right),
\]
and using the Einstein condition we conclude that $\langle B_3, X_j^0\rangle = 0$. Analogously, we obtain that $\langle B_i, X_j^0\rangle = 0,$ for all $i=1,2,3$, $j=1,\ldots,d$, thus $\langle \slg_2(\RR), \mg_0\rangle = 0$. Therefore, $\slg_2(\RR)\perp \mg$, and this implies that the metric is locally a Riemannian product. But this is a contradiction, since $\widetilde{\Sl_2(\RR)}$ does not admit any left-invariant Einstein metric.
\end{proof}

Finally, we study the family $\left(\Sl_2(\RR) \times \Sl_2(\CC)\right)/\Delta_{p,q}\U(1)$. Unfortunately, we are not able to deal with the case where $p = q$. Notice though that this missing case represents just one homogeneous space from the above infinite family.

With respect to the inclusions $\hg_1:= \sog(2) \subseteq \slg_2(\RR), \hg_2 := \ug(1) \subseteq \sug(2) \subseteq \slg_2(\CC)$ we have that
\[
    \hg := \Delta_{p,q} \ug(1) \subseteq \hg_1 \oplus \hg_2 \subseteq \slg_2(\RR) \oplus \slg_2(\CC) =: \ggo.
\]
Given an $\Ad(H)$-invariant inner product on some reductive complement $\mg$, we extend it in the usual way to an $\Ad(H)$-invariant inner product $\ip$ on $\ggo$. By looking at the decomposition of the isotropy representation from Table \ref{tabla} in the case when $p\neq q$,
\[
    \mg = \qg_0^{(1)} \oplus \qg_1^{(2)} \oplus \pg_0^{(1)} \oplus \pg_1^{(2)} \oplus \pg_2^{(2)}, \qquad \qg_1 \simeq \pg_1 \not\simeq \pg_2,
\]
we see that the ideals $\slg_2(\RR), \slg_2(\CC)$ are orthogonal (notice that $\pg_2^{(2)}$, $\qg_1^{(2)} \oplus \pg_0^{(1)} \oplus \pg_1^{(2)}$ correspond to reductive complements for the homogeneous spaces $\Sl_2(\RR)/\SO(2)$, $\Sl_2(\CC)/\U(1)$, respectively, and $\hg \oplus \qg_0 = \hg_1 \oplus \hg_2$). This easily gives that $\ip$ is $\ad(\hg)$-invariant if and only if it is both $\ad(\hg_1)$- and $\ad(\hg_2)$-invariant. Thus, $\qg_0$ acts by skew-symmetric endomorphisms on $\ggo$, and by Lemma \ref{lem_formulaRicci} we have that
\[
    \Ricci(Y,Y) = \unc \sum_{i,j} \langle [X_i,X_j]_\mg, Y \rangle^2 \geq 0, \qquad Y\in \qg_0.
\]

\begin{remark}
It is worth pointing out that when $p\neq q$, all homogeneous metrics in the above homogeneous space can be approximated by \emph{strictly locally homogeneous metrics} (cf.~ \cite{Tric92}), namely, metrics which are locally homogeneous but are not locally isometric to any globally homogeneous manifold. This is simply done by considering irrational slopes approximating the rational slope $p/q$ of the given space. It was proved in \cite{Spiro} (see also \cite{Bhm15}) that strictly locally homogeneous metrics do not have non-positive Ricci curvature.

On the other hand, if $p=q$ there exist homogeneous metrics that cannot be approximated in that way.
\end{remark}

\section{Non-unimodular transitive group}\label{sectionnonuni}

In this section we study Einstein homogeneous spaces $G/K$ of negative scalar curvature with $G$ non-unimodular and $G/K$ as in Theorem \ref{structure}. Following the discussion of Section~ \ref{prelimstruct} we can assume that
\begin{equation}\label{decom}
\ggo=(\ggo_1 + \zg(\ug)) \ltimes \ngo,
\end{equation} where $\ug=\ggo_1 + \zg(\ug)$ is a reductive Lie algebra, $\ggo_1$ is semisimple with no compact ideals, $\kg \subset \ggo_1$ and $\zg(\ug)=\RR H ,$ with $H$ the mean curvature vector (see Corollary \ref{cor_rankone}).

Before starting the proof of our main result in the non-unimodular case, we state two lemmas which yield information about semisimple homogeneous spaces in low dimensions. Their proof follows immediately from Table \ref{tabla}, and the fact that irreducible symmetric spaces of the non-compact type are diffeomorphic to Euclidean spaces.

\begin{lemma}\label{diff5}
Let $G_1/K$ be a simply-connected semisimple homogeneous space of the non-compact type\footnote{see Definition \ref{defsshomogspace}.} with $n = \dim G_1/K  \leq 5.$ Then, either $G_1/K=\Sl_2(\CC)/\U(1)$ or $G_1/K \simeq \RR^n$.
\end{lemma}

\begin{lemma}\label{diff6}
Let $G_1/K$ be a $6$-dimensional simply-connected semisimple homogeneous space of the non-compact type such that $G_1$ contains $\widetilde{\Sl_2(\RR)}$ as a simple factor. Then, $G_1/K$ is diffeomorphic to~ $\RR^6.$
\end{lemma}

We now focus on the proof of Theorem \ref{mainnonuni}. Given a homogeneous Einstein space $(G/K,g)$ with negative scalar curvature and $G$ non-unimodular, we consider for it the decomposition given in \eqref{decom}. By virtue of Corollary \ref{reductionG1/K}, our goal will be to prove that $G_1/K$ is diffeomorphic to a Euclidean space. Observe that we may assume $\dim G_1/K \leq \dim G/K -3.$ Indeed, we always have $\dim \zg(\ug) =1$ and $\dim \ngo \geq \dim \zg(\ug)$ because the representation $\theta|_{\zg(\ug)} : \zg(\ug) \to \ngo$ is faithful (its kernel must be in the nilradical). If $\dim \ngo = 1$, we know by Theorem \ref{thm_lemadimn} and Remark~ \ref{rmkn1} that $(G/K, g)$ is a Riemannian product, and thus not de Rham irreducible. More generally, for this reason we can also assume that $\theta|_{\ggo_1} \neq 0.$  We now proceed with the proof, considering different cases according to the dimension of $G/K.$

\subsection{$\dim G/K \leq 8$}\label{dimG/Kleq8} $ $

We have that $\dim G_1/K \leq 5.$ By Lemma \ref{diff5}, either $G_1/K$ is diffeomorphic to $\RR^n$ for some $n \leq 5,$ or $G_1/K=\Sl_2(\CC)/\U(1).$ In the latter case, $\theta|_{\ggo_1}=0,$ since $\dim \ngo=2$ and there exists no nontrivial $2$-dimensional representation of the simple Lie algebra $\slg_2(\CC)$. Thus, this is a product case.

\subsection{$\dim G/K=9$} $ $


\subsubsection{$\dim G_1/K=6$}\label{G1/K=6} As $\dim \ngo=2$ we know that $\theta(\ggo_1)\subseteq \End(\RR^2)$ is semisimple, so $\ggo_1$ must have an ideal isomorphic to $\slg_2(\RR).$ We conclude that $G_1/K$ is diffeomorphic to $\RR^6$ by using Lemma \ref{diff6}.

\subsubsection{$\dim G_1/K \leq 5$} By Lemma \ref{diff5}, we only consider the case where $G_1/K=\Sl_2(\CC)/\U(1).$ We have that $\dim \ngo = 3$, and it is easy to see that $\theta|_{\ggo_1}=0$ since there is no subalgebra of $ \slg_3(\RR)$ isomorphic to $\slg_2(\CC)$. Hence this is also a product case.

\subsection{$\dim G/K=10$}

\subsubsection{$\dim G_1/K=7$} We have that $\dim \ngo=2$ and $\theta(\ggo_1)$ is semisimple, that is, $\ggo_1$ has an ideal isomorphic to $\slg_2(\RR)$. The list of all possible homogeneous spaces $G_1/K$ to consider is very long. However, if $G_1/K$ is a product of lower dimensional homogeneous spaces, then it must be a product of some irreducible symmetric spaces of non-compact type and some of the spaces in Table \ref{tabla}. Since its dimension is $7$, it easily follows that there is at most one factor which is non-symmetric, and Proposition \ref{prodRiem} implies that on $G_1/K$ the metric is a product of corresponding invariant metrics on each of the factors. Moreover, since the kernel of $\theta|_{\ggo_1}$ has codimension $3$, $\theta$ must necessarily vanish on some of these factors. This implies at once that the whole space $G/K$ splits as a Riemannian product.

By the preceeding discussion, we may now assume that $G_1/K$ is non-product, i.e.~ it is one of the spaces listed in Table \ref{tabla}. Since $\ggo_1$ has an ideal isomorphic to $\slg_2(\RR)$ there are actually only two possibilities: the simply connected covers of $\Sl_2(\RR)^3/\Delta T^2_{a,b,c}$ and of $\SU(2,1) \times \Sl_2(\RR)/\Delta_{p,q}\U(1)(\SU(2)\times \{e\})$. But both of them are diffeomorphic to $\RR^7$, hence we are done.

\subsubsection{$\dim G_1/K=6$} 


Here $\dim \ngo=3$ and $\theta|_{\ggo_1} \neq 0.$ Since $\theta|_{\ggo_1}$ maps $\ggo_1$ into $\slg_3(\RR)$, for it to be non-trivial it is necessary that $\ggo_1$ contains at least one simple ideal of dimension at most $8$. By Lemma \ref{diff6}, we may further assume that $G_1$ contains no $\widetilde{\Sl_2(\RR)}$ factor. Thus it is clear from Table \ref{tabla} that if $G_1/K$ were a product of lower dimensional homogeneous spaces then each factor would be symmetric space, and the result would follow. 

On the other hand, if $G_1/K$ is non-product, it also follows from Table \ref{tabla} that the only possibilities are
\begin{align*}
\SU(2,1)/T_{max}, \quad \Sl_2(\CC).
\end{align*}
But then we must have $\theta|_{\ggo_1} = 0$, because there exist no nontrivial $3$-dimensional representations of the simple Lie algebras $\sug(2,1)$ or $\slg_2(\CC)$.


\subsubsection{$\dim G_1/K \leq 5$}\label{section514sl2C} By using Lemma \ref{diff5} we have that either $G_1/K \simeq \RR^n,$ for some $n \leq 5,$ or $G_1/K=\Sl_2(\CC)/\U(1).$ We need to analyze the latter case. 
We are reduced to showing that equation \eqref{eqRicU/K} has no solutions for $\Sl_2(\CC)$-invariant metrics on $\Sl_2(\CC)/\U(1)$, for any $\theta: \ggo_1 \to \End(\RR^4)$ that satisfies \eqref{eqmmtheta}. We now use the notation of Section \ref{sectionsl2C}. Let us assume that $\theta(\ggo_1) \neq 0$, since otherwise we are in a product case, and consider an arbitrary $\Ad(\U(1))$-invariant inner product on $\slg_2(\CC)$ written in the form $\ip_h$ given in Lemma \ref{lemmasl2C}, $h\in \Gl_5(\RR)$. An orthonormal basis for the reductive complement $\mg$ is given by $\mathcal{B}_h = \{h^{-1} Y_0, \ldots, h^{-1} X_2 \}$. Up to equivalence of representations, there are two $4$-dimensional real faithful representations of $\slg_2(\CC)$: the tautological representation, and its conjugate, and they are both irreducible. Let us consider the case where $\theta$ is equivalent to the tautological representation (the other case is completely analogous). There exists $h_2 \in \Gl(\ngo)$ such that with respect to the inner product $h_2 \cdot \ip_\ngo$ the matrices of $\theta(Z)$, \ldots, $\theta(X_2)$ have the forms \eqref{matricessl2C} (after identifying a complex number $a+b i$ with a $2\times 2$ real matrix $\minimatrix{a}{b}{-b}{a}$). This is equivalent to saying that the matrices of $(h_2^{-1} \cdot \theta)(Z), \ldots, (h_2^{-1} \cdot \theta)(X_2)$ have such a form with respect to the inner product $\ip_\ngo$. But now an easy computation shows that
\[
    \sum_{Y \in \, \mathcal{B}_h} \left[\left(h_2^{-1} \cdot \theta\right)(Y), \left(h_2^{-1} \cdot \theta\right)(Y)^t \right] = 0,
\]
that is, $h_2^{-1}\cdot\, \theta$ is also a zero of the moment map for the natural $\Gl(\ngo)$-action on $\End(\slg_2(\CC),\End(\ngo))$ (see Remark \ref{remarks}, \eqref{remarktheta}). From the rigidity imposed by Geometric Invariant Theory for such zeros \cite[Theorem 4.3]{RS90}, we can conclude that in fact $h_2\in \Or\left(\ngo, \ip_\ngo\right)$, and thus the matrices of $\theta(Z), \ldots, \theta(X_2)$ have the form \eqref{matricessl2C} with respecto to $\ip_\ngo$ (see also the proof of \cite[Proposition A.1]{semialglow} for a more detailed application of this argument). We now plug this information into equation \eqref{eqRicU/K}, and since $\theta(X_2)$ is skew-symmetric, by looking at the Ricci curvature and using Lemma \ref{lemmasl2C} we obtain that $d=0$. In other words, $h$ is diagonal, and in particular the metric associated to $\ip_h$ leaves orthogonal the Cartan decomposition $\slg_2(\CC) = \kg \oplus \pg$, $\kg = \hg \oplus \qg_1$, $\pg = \pg_0 \oplus \pg_1$. Notice that the operator $C_\theta \in \End(\mg)$ given by
\[
    \left\langle C_\theta X, Y\right\rangle = \tr S \left( \theta(X)\right) S\left(\theta(Y)\right)
\]
is a positive multiple of the identity on $\pg$. Hence, by following the same arguments used in the proof of \cite[Theorem 1]{Nkn2} we can conclude that the equation
\begin{equation}\label{eqRictheta}
    \Ricci_{\ip} = c \, I  + C_\theta
\end{equation}
can not be satisfied for $\ip_h$. Therefore, there are no Einstein metrics in this case.

\begin{remark}
The fact that the proof of \cite[Theorem 1]{Nkn2} could be adapted for the more general equation \eqref{eqRicU/K} as long as $G_1$ is simple was kindly communicated to us by Jorge Lauret \cite{Lauretpersonalcom}.
\end{remark}



\section{Strong Alekseevskii's conjecture}\label{strong}

This section is devoted to studying the strong Alekseevskii conjecture and showing that it holds up to dimension $8$, with the possible exceptions of invariant metrics on non-compact semisimple Lie groups or on the space $\left(\Sl_2(\RR)\times \Sl_2(\CC)\right)/\Delta\U(1)$.

As in the previous section, we consider Einstein homogeneous spaces $G/K$ with $G$ chosen as in Theorem~ \ref{structure}, which are not Riemannian products. We also assume the simply-connected hypothesis, which is non-restrictive since it suffices to prove the strong Alekseevskii conjecture in the simply-connected case (see Remark \ref{remarks} (e) and \cite{AC99, Jab15}). Regarding the semisimple case, according to Theorem \ref{thmsemisimple} all semisimple homogeneous spaces in Table \ref{tabla} are either Lie groups, or the space $\left(\Sl_2(\RR)\times \Sl_2(\CC)\right)/\Delta\U(1)$, or they do not admit an Einstein metric. The remaining semisimple homogeneous spaces are symmetric spaces, and it is well-known that they are isometric to solvmanifolds. Therefore, by \cite{Dtt88}, in the following we will only study the cases where the transitive group is non-unimodular (see Remark \ref{remarks}, \eqref{rmkDotti}). We proceed case by case, according to the dimension of $G/K.$ Our goal will be to show that $G_1/K$ is isometric a solvmanifold.

\subsection{$\dim G/K=6$}\label{strong6} $ $

These spaces were analyzed by Jablonski and Petersen in \cite[\S 4]{JblPtr}.

\subsection{$\dim G/K=7$} $ $

First we state the following lemma, which follows easily from Table \ref{tabla} and the well-known fact that irreducible symmetric spaces of the non-compact type are isometric to solvmanifolds.

\begin{lemma}\label{solv5}
Let $(G_1/K, g)$ be a simply-connected semisimple homogeneous space of the non-compact type with a $G_1$-invariant metric $g$, and $\dim G_1/K \leq 5.$ Then, either $(G_1/K,g)$ is isometric to a solvmanifold, or $G_1/K$ is one of the following spaces
\begin{gather*}
    \widetilde{\Sl_2(\RR)}, \quad \Sl_2(\CC)/\U(1),\quad \left(\Sl_2(\RR)\times \Sl_2(\RR)\right)/\Delta_{p,q} \SO(2), \\
     \SU(2,1)/\SU(2), \quad \Sl_2(\RR) \times \Sl_2(\RR)/ \SO(2).
\end{gather*}
\end{lemma}


By using Theorem \ref{thm_lemadimn}, Remark \ref{rmkn1} and Corollary \ref{cor_rankone}, we can assume that $\dim G_1/K \leq 4,$ $\dim\zg(\ug)=1$ and $\dim \ngo \geq 2.$ We divide into cases according to the dimension of $G_1/K.$

\subsubsection{$\dim G_1/K=4$} By Lemma \ref{solv5} we know that $G_1/K$ is a solvmanifold.

\subsubsection{$\dim G_1/K=3$}\label{G/K7G_1/K3} The only case in which $G_1/K$ is not a solvmanifold is when $G_1/K=\widetilde{\Sl_2(\RR)}.$ 
The existence of an example in this case would imply that there is a $6$-dimensional unimodular expanding algebraic soliton, by using that non-unimodular Einstein spaces are one-dimensional extensions of unimodular algebraic solitons (see \cite[\S 6]{alek}). Since for that soliton one would have $\ug=\slg_2(\RR)$, we arrive at a contradiction by using \cite[Appendix]{semialglow}.

\subsection{$\dim G/K=8$} $ $

We assume that $\dim G_1/K \leq 5$ and $\dim \ngo \geq 2.$

\subsubsection{$\dim G_1/K=5$} We have that $\dim \ngo=2$. Then, by using Lemma \ref{solv5}, we consider the following possibilities to $G_1/K:$
\[
\Sl_2(\CC)/\U(1), \quad \SU(2,1)/\SU(2), \quad \left(\Sl_2(\RR)\times \Sl_2(\RR)\right)/\Delta_{p,q}\SO(2), \quad \Sl_2(\RR) \times \Sl_2(\RR)/ \SO(2).
\]
In the first two cases, we must have $\theta|_{\ggo_1} = 0$ since there exist no nontrivial $2$-dimensional representations of the simple Lie algebras $\slg_2(\CC)$ and $\sug(2,1).$ By Theorem \ref{thm_lemadimn} and Remark \ref{rmkn1}, these are product cases. In the latter case, we also are in a product case. Indeed, the metric restricted to $G_1/K$ is a product metric. In addition, since $\dim \ngo=2,$ $\theta$ must necessarily vanish on some of these factors. This implies at once that the whole space $G/K$ splits as a Riemannian product. We now deal with the case $G_1/K =\left(\Sl_2(\RR)\times \Sl_2(\RR)\right)/\Delta_{p,q}\SO(2)$.

This case is similar in nature to the one in Section \ref{section514sl2C}: we are reduced to solving equation~ \eqref{eqRicU/K} for invariant metrics on $G_1/K$, for every possible representation $\theta: \ggo_1 = \slg_2(\RR)\oplus \slg_2(\RR) \to \End(\RR^2)$. Let $H = \minimatrix{1}{0}{0}{-1}$, $X = \minimatrix{0}{1}{-1}{0}$, $Y = \minimatrix{0}{1}{1}{0}$ be a basis for $\slg_2(\RR)$, and consider the ordered basis $\mathcal{B}$ for $\slg_2(\RR) \oplus \slg_2(\RR)$ given by
\begin{align*}
    Z = \left(p \,H, q\, H\right), \quad X_0 &=  \left(q \,H, -p\, H\right), \quad Y_1 = \left( X, 0\right), \\
    Y_2 = \left( Y,0\right), \quad X_1 &= \left( 0, X\right), \quad X_2 = \left( 0, Y\right).
\end{align*}
The isotropy subalgebra is $\hg_{p,q} = \RR Z$, $\mg = \operatorname{span_\RR}\{X_0, Y_1, Y_2, X_1, X_2 \}$ is a reductive complement, and the decomposition into irreducible submodules is given by $\mg = \qg_0 \oplus \pg_1 \oplus \pg_2$, where $\qg_0 = \RR X_0$, $\pg_1 = \RR Y_1 \oplus \RR Y_2$, $\pg_2 = \RR X_1 \oplus \RR X_2$, and $\pg_1 \simeq \pg_2$ if and only if $p=q$.

First notice that $\theta$ must have a kernel, which we may assume without loss of generality to be the first $\slg_2(\RR)$ factor. This implies that $\theta(H,0) = 0$, and since $\theta(Z)$ is skew-symmetric, it is clear that $\theta(X_0)$ must also be skew-symmetric. Thus, using equation \eqref{eqRicU/K} we obtain that
\[
    \Ricci_{G_1/K}(X_0,X_0) < 0.
\]
This is already enough to rule out the cases $p\neq q$, since in those cases we have $\pg_1 \not\simeq \pg_2$, which forces $X_0$ to act skew-symmetrically on $\ggo_1$, and by Lemma \ref{lem_formulaRicci} we get a contradiction.

Let us now consider the remaining case $p = q = 1$, which is considerably more difficult. The following is the analogous of Lemma \ref{lemmasl2C} for this situation, and can be proved in the very same way.

\begin{lemma}
Up to isometry, $\Sl_2(\RR)\times \Sl_2(\RR)$-invariant metrics on $\left(\Sl_2(\RR)\times \Sl_2(\RR)\right) / \Delta_{1,1} \SO(2)$ can be parameterized by $\Ad(\SO(2))$-invariant inner products on $\mg$ of the form
\[
    \langle \cdot, \cdot \rangle_h = \langle h\, \cdot\, , h\, \cdot \rangle_0,
\]
where $h\in \Gl_5(\RR)$ is given by
\[
    h = \left[\begin{array}{ccccc} e &0 &0 &0 & 0\\ 0& a &0 &0 &0 \\0 & 0& a & 0&0 \\0 &d & 0& b &0 \\ 0&0 &d &0 & b \end{array}\right], \qquad a,b,d,e \in \RR, \quad a,b,e\neq 0.
\]
\end{lemma}

Reasoning as in Section \ref{section514sl2C}, we see that condition \eqref{eqmmtheta} implies that $\theta$ restricted to the second $\slg_2(\RR)$ factor is nothing but the tautological representation of this Lie algebra. Let us consider the operator $\Ricci^\theta \in \End(\mg)$ given by
\[
    \left\langle \Ricci^\theta  \,  X,   Y\right\rangle_h = \Ricci_{G_1/K}(X,Y) - \tr S(\theta(X)) S(\theta(Y)).
\]
Equation \eqref{eqRicU/K} for a metric $\ip_h$ can be rephrased as
\begin{equation}\label{eqRiccithetaop}
    \Ricci^\theta = c I, \qquad c<0.
\end{equation}
Since we now know $\theta$ and $\ip_h$ explicitly, we can actually compute $\Ricci^\theta$ in terms of $a,b,d,e$. Let us call $r_{i,j}^\theta$, $1\leq i,j \leq 5$, the entries of the matrix of $\Ricci^\theta$ with respect to the $\ip_h$-orthonormal ordered basis $\mathcal{B}_h = \left\{ h^{-1} Y \mid Y\in \mathcal{B} \right\}$ . Then, assuming that $\det h = 1$, we have that
\begin{align*}
    r_{1,1}^\theta &=  \unm \left( a^4 e^4 + \left(b^2 - d^2\right)^2  \right) + a^2 d^2 \left( e^2- 4b^2\right) \left( e^2 + 4 b^2\right), \\
    r_{1,1}^\theta + 2\,  r_{4,4}^\theta + 2\, a \cdot d^{-1} \, r_{2,4}^\theta & = \unm \left(a^2 - b^2 + d^2 \right)^2 e^4 + 4 a^4 b^2 \left( 4 b^2 - e^2  \right).
\end{align*}
Despite the ugliness of these formulas, we see that all the terms and factors on the right hand side are positive except for $ e^2 - 4 b^2 $, which appears with a different sign in both of them. For a solution of \eqref{eqRiccithetaop}, both expressions should be negative (they would equal $c$ and $3\, c$, respectively). It is now clear that such a solution does not exist.


\subsubsection{$\dim G_1/K=4$} By Lemma \ref{solv5} we know that $G_1/K$ is isometric to a solvmanifold.

\subsubsection{$\dim G_1/K=3$} Here, the only case to consider is $G_1/K=\widetilde{\Sl_2(\RR)}.$ 
Similarly to the case in Section~ \ref{G/K7G_1/K3}, we have a contradiction.

\bibliography{aleklow}

\renewcommand{\MR}[1]{} \newcommand{\noop}[1]{} \def\cprime{$'$}
\providecommand{\bysame}{\leavevmode\hbox to3em{\hrulefill}\thinspace}
\providecommand{\MR}{\relax\ifhmode\unskip\space\fi MR }
\providecommand{\MRhref}[2]{%
  \href{http://www.ams.org/mathscinet-getitem?mr=#1}{#2}
}
\providecommand{\href}[2]{#2}
\begin{thebibliography}{HPW14}

\bibitem[AC99]{AC99}
Dmitri Alekseevski{\u\i} and Vicente Cort{\'e}s, \emph{Isometry groups of
  homogeneous quaternionic {K}\"ahler manifolds}, J. Geom. Anal. \textbf{9}
  (1999), no.~4, 513--545. \MR{1757577 (2001f:53094)}

\bibitem[ADF96]{AleDotFer}
Dmitri Alekseevski{\u\i}, Isabel Dotti, and Carlos Ferraris, \emph{Homogeneous
  {R}icci positive 5-manifolds}, Pacific J. Math. \textbf{175} (1996), 1--12.

\bibitem[AK75]{AlkKml}
Dmitri Alekseevski{\u\i} and Boris~N. Kimel{\cprime}fel{\cprime}d,
  \emph{Structure of homogeneous {R}iemannian spaces with zero {R}icci
  curvature}, Funktional. Anal. i Prilov Zen. \textbf{9} (1975), no.~2, 5--11.

\bibitem[AL15]{semialglow}
Romina~M. Arroyo and Ramiro Lafuente, \emph{Homogeneous ricci solitons in low
  dimensions}, International Mathematics Research Notices \textbf{2015} (2015),
  no.~13, 4901--4932.

\bibitem[Ale12]{Al11}
Dmitri Alekseevsky, \emph{Homogeneous {L}orentzian manifolds of a semisimple
  group}, J. Geom. Phys. \textbf{62} (2012), no.~3, 631--645. \MR{2876787}

\bibitem[AW75]{AW75}
Simon Aloff and Nolan~R. Wallach, \emph{An infinite family of distinct
  {$7$}-manifolds admitting positively curved {R}iemannian structures}, Bull.
  Amer. Math. Soc. \textbf{81} (1975), 93--97. \MR{0370624 (51 \#6851)}

\bibitem[Bes87]{Bss}
Arthur~L. Besse, \emph{Einstein manifolds}, Ergebnisse der Mathematik und ihrer
  Grenzgebiete (3) [Results in Mathematics and Related Areas (3)], vol.~10,
  Springer-Verlag, Berlin, 1987.

\bibitem[BK06]{BhmKrr}
Christoph B{\"o}hm and Megan~M. Kerr, \emph{Low-dimensional homogeneous
  {E}instein manifolds}, Trans. Amer. Math. Soc. \textbf{358} (2006), no.~4,
  1455--1468. \MR{2186982 (2006g:53056)}

\bibitem[Boc48]{Bochner1948}
S.~Bochner, \emph{Curvature and {B}etti numbers}, Ann. of Math. (2) \textbf{49}
  (1948), 379--390. \MR{0025238 (9,618d)}

\bibitem[B{\"o}h14]{Bhm15}
Christoph B{\"o}hm, \emph{On the long time behavior of homogeneous {R}icci
  flows}, preprint, 2014.

\bibitem[Car27]{Car27}
Elie Cartan, \emph{Sur certaines formes {R}iemanniennes remarquables des
  g\'eom\'etries \`a groupe fondamental simple}, Ann. Sci. \'Ecole Norm. Sup.
  (3) \textbf{44} (1927), 345--467. \MR{1509283}

\bibitem[DM88]{Dtt88}
Isabel Dotti~Miatello, \emph{Transitive group actions and ricci curvature
  properties.}, Michigan Math. J \textbf{35} (1988), no.~3, 427--434.

\bibitem[FC14]{FC13}
Edison Fernandez-Culma, \emph{Classification of {N}ilsoliton metrics in
  dimension seven}, J. Geom. Phys. \textbf{86} (2014), 164--179.

\bibitem[Heb98]{Heb}
Jens Heber, \emph{Noncompact homogeneous {E}instein spaces}, Invent. Math.
  \textbf{133} (1998), no.~2, 279--352.

\bibitem[Hel78]{Helgason}
Sigurdur Helgason, \emph{Differential geometry, {L}ie groups, and symmetric
  spaces}, Pure and Applied Mathematics, vol.~80, Academic Press, Inc.
  [Harcourt Brace Jovanovich, Publishers], New York-London, 1978. \MR{514561
  (80k:53081)}

\bibitem[HPW14]{HePtrWyl}
Chenxu He, Peter Petersen, and William Wylie, \emph{Warped product {E}instein
  metrics on homogeneous spaces and homogeneous {R}icci solitons}, J. Reine
  Angew. Math. (2014), in press.

\bibitem[Jab14]{Jbl13b}
Michael Jablonski, \emph{Homogeneous {R}icci solitons are algebraic}, Geom.
  Topol. \textbf{18} (2014), no.~4, 2477--2486. \MR{3268781}

\bibitem[Jab15a]{Jbl}
\bysame, \emph{Homogeneous {R}icci solitons}, J. Reine Angew. Math.
  \textbf{699} (2015), 159--182. \MR{3305924}

\bibitem[Jab15b]{Jab15}
\bysame, \emph{Strongly solvable spaces}, Duke Math. J. \textbf{164} (2015),
  no.~2, 361--402. \MR{3306558}

\bibitem[Jen69]{Jns}
Gary~R. Jensen, \emph{Homogeneous {E}instein spaces of dimension four}, J.
  Differential Geometry \textbf{3} (1969), 309--349.

\bibitem[JP14]{JblPtr}
Michael Jablonski and Peter Petersen, \emph{A step towards the alekseevskii
  conjecture}, arXiv:1403.5037v2, \noop{1}(2014).

\bibitem[Lau02]{finding}
Jorge Lauret, \emph{Finding einstein solvmanifolds by a variational method},
  Math. Z. \textbf{241} (2002), 83--99.

\bibitem[Lau09]{cruzchica}
\bysame, \emph{Einstein solvmanifolds and nilsolitons}, New developments in
  {L}ie theory and geometry, Contemp. Math., vol. 491, Amer. Math. Soc., 2009,
  pp.~1--35.

\bibitem[Lau11]{solvsolitons}
\bysame, \emph{Ricci soliton solvmanifolds}, J. Reine Angew. Math. \textbf{650}
  (2011), 1--21.

\bibitem[Lau12]{Lauretpersonalcom}
\bysame, \emph{Personal communication}, 2012.

\bibitem[LL14]{alek}
Ramiro Lafuente and Jorge Lauret, \emph{Structure of homogeneous {R}icci
  solitons and the {A}lekseevskii conjecture}, J. Differential Geom.
  \textbf{98} (2014), no.~2, 315--347. \MR{3263520}

\bibitem[LW99]{LeBWang}
Claude LeBrun and McKenzie Wang (eds.), \emph{Surveys in differential geometry:
  essays on {E}instein manifolds}, Surveys in Differential Geometry, VI,
  International Press, Boston, MA, 1999, Lectures on geometry and topology,
  sponsored by Lehigh University's Journal of Differential Geometry.
  \MR{1798603 (2001f:53003)}

\bibitem[Mil76]{Mln}
John Milnor, \emph{Curvatures of left-invariant metrics on lie groups}, Adv.
  Math. \textbf{21} (1976), 293--329.

\bibitem[Nik00]{Nkn2}
Yurii Nikonorov, \emph{On the {R}icci curvature of homogeneous metrics on
  noncompact homogeneous spaces}, Sibirsk. Mat. Zh. \textbf{41} (2000), no.~2,
  421--429, iv.

\bibitem[Nik04]{Nkn04}
\bysame, \emph{Compact homogeneous {E}instein 7-manifolds}, Geom. Dedicata
  \textbf{109} (2004), no.~1, 7--30.

\bibitem[Nik05]{Nkn1}
\bysame, \emph{Noncompact homogeneous {E}instein 5-manifolds}, Geom. Dedicata
  \textbf{113} (2005), 107--143.

\bibitem[NN06]{NikitenkoNikonorov}
Evgenii~Vital'evich Nikitenko and Yurii Nikonorov, \emph{Six-dimensional
  {E}instein solvmanifolds [{T}ranslation of {M}at. {T}r. {\bf 8} (2005), no.
  1, 71--121; mr1955023]}, Siberian Adv. Math. \textbf{16} (2006), no.~1,
  66--112. \MR{2206760}

\bibitem[NR03]{NknRod03}
Yurii Nikonorov and Evgenii~Dmitrievich Rodionov, \emph{Compact homogeneous
  {E}instein 6-manifolds}, Diff. Geom. Appl. \textbf{19} (2003), no.~3,
  369--378.

\bibitem[RS90]{RS90}
Roger~Wolcott Richardson and Peter Slodowy, \emph{Minimum vectors for real
  reductive algebraic groups}, J. London Math. Soc. (2) \textbf{42} (1990),
  no.~3, 409--429. \MR{1087217 (92a:14055)}

\bibitem[Spa11]{Sparkssurvey}
James Sparks, \emph{Sasaki-{E}instein manifolds}, Surveys in differential
  geometry. {V}olume {XVI}. {G}eometry of special holonomy and related topics,
  Surv. Differ. Geom., vol.~16, Int. Press, Somerville, MA, 2011, pp.~265--324.
  \MR{2893680 (2012k:53082)}

\bibitem[Spi93]{Spiro}
Andrea Spiro, \emph{A remark on locally homogeneous {R}iemannian spaces},
  Results Math. \textbf{24} (1993), no.~3-4, 318--325. \MR{1244285 (94m:53071)}

\bibitem[Tri92]{Tric92}
Franco Tricerri, \emph{Locally homogeneous {R}iemannian manifolds}, Rend. Sem.
  Mat. Univ. Politec. Torino \textbf{50} (1992), no.~4, 411--426 (1993),
  Differential geometry (Turin, 1992). \MR{1261452 (94k:53065)}

\bibitem[Wan82]{Wng82}
McKenzie~Y. Wang, \emph{Some examples of homogeneous {E}instein manifolds in
  dimension seven}, Duke Math. J. \textbf{49} (1982), no.~1, 23--28. \MR{650366
  (83k:53069)}

\bibitem[Wan12]{Wang2012}
\bysame, \emph{Einstein metrics from symmetry and bundle constructions: a
  sequel}, Differential Geometry: Under the Influence of S.-S. Chern, Advanced
  Lectures in Mathematics, vol.~22, Higher Education Press/International
  Press., Beijing-Boston, 2012, pp.~253--309.

\bibitem[Wil03]{Wll03}
Cynthia Will, \emph{Rank-one einstein solvmanifolds of dimension 7}, Diff.
  Geom. Appl. \textbf{19} (2003), 307--318.

\end{thebibliography}
\bibliographystyle{amsalpha}

%

\end{document}